\documentclass[11pt]{amsart}
\usepackage{amsfonts, amsmath, amssymb, amsthm, stmaryrd, color, enumerate, relsize}
\usepackage[pdftex]{graphicx}
\usepackage{mathrsfs,array}
\usepackage{xy}
\usepackage{hyperref}
\usepackage{tikz-cd}
\input xy
\xyoption{all}
\usepackage{scalerel, stackengine}
\stackMath
\newcommand\widecheck[1]{%
\savestack{\tmpbox}{\stretchto{%
  \scaleto{%
    \scalerel*[\widthof{\ensuremath{#1}}]{\kern-.6pt\bigwedge\kern-.6pt}%
    {\rule[-\textheight/2]{1ex}{\textheight}}
  }{\textheight}%
}{0.5ex}}%
\stackon[1pt]{#1}{\scalebox{-1}{\tmpbox}}%
}
\DeclareMathOperator{\sHom}{\!{\mathscr{H}\! om}}
\newcommand{\dotimes}{\otimes^{\mathbb L}}

\numberwithin{equation}{section}
\newtheorem{theorem}{Theorem}
\numberwithin{theorem}{section}

\newtheorem{corollary}[theorem]{Corollary}
\newtheorem{proposition}[theorem]{Proposition}

\newtheorem{definition}[theorem]{Definition}

\theoremstyle{definition}
\newtheorem{remark}[theorem]{Remark}

\newcommand{\Hom}{\mathrm{Hom}}
\newcommand{\BC}{\mathcal{BC}}
\newcommand{\Bun}{\mathrm{Bun}}
\newcommand{\Div}{\mathrm{Div}}
\newcommand{\Spd}{\mathrm{Spd}}
\newcommand{\Spa}{\mathrm{Spa}}
\newcommand{\Hloc}{{\mathcal H}{\mathrm{ck}}}

\date{\today}

\title[Geometrization of local Langlands, Motivically]{Geometrization of the local Langlands correspondence, Motivically}
\author{Peter Scholze}
\begin{document}

\begin{abstract} Based on the formalism of rigid-analytic motives of Ayoub--Gallauer--Vezzani \cite{AyoubGallauerVezzani}, we extend our previous work \cite{FarguesScholze} from $\ell$-adic sheaves to motivic sheaves. In particular, we prove independence of $\ell$ of the $L$-parameters constructed there.
\end{abstract}

\maketitle

\tableofcontents

\section{Introduction}

Let $E$ be a nonarchimedean local field with residue field $\mathbb F_q$ of characteristic $p$, and $G$ a reductive group over $E$. The local Langlands correspondence seeks to associate to an irreducible complex smooth representation $\pi$ of $G(E)$ an $L$-parameter, which is a $\widehat{G}(\mathbb C)$-conjugacy class of continuous $1$-cocycles
\[
W_E\to \widehat{G}(\mathbb C)
\]
where $\widehat{G}$ is the Langlands dual group, endowed with an action of $W_E$. We will restrict ourselves in this paper to semisimple $L$-parameters, which are those whose $\widehat{G}(\mathbb C)$-conjugacy class is closed; any $L$-parameter admits a semisimplification (by passing to the minimal orbit in its closure). In \cite{FarguesScholze}, such semisimple $L$-parameters are constructed for arbitrary $G$, but the construction uses \'etale cohomology and hence as usual requires the choice of an auxiliary prime $\ell\neq p$ and deals with $\overline{\mathbb Q}_\ell$-representations (and also requires a choice of $\sqrt{p}\in \overline{\mathbb Q}_\ell$). Thus, for any isomorphism $\iota: \overline{\mathbb Q}_\ell\cong \mathbb C$, we can take $\sqrt{p}\in \overline{\mathbb Q}_\ell$ as the preimage under $\iota$ of the positive root, and get a map
\[
\pi\mapsto \varphi_{\pi,\iota}
\]
from irreducible complex smooth representations $\pi$ of $G(E)$ towards semisimple $L$-parameters. The goal of this paper is to prove the following result, giving canonical semisimple $L$-parameters for complex representations of $G(E)$.

\begin{theorem}\label{thm:A} The association $\pi\mapsto \varphi_{\pi,\iota}$ is independent of $\ell$ and $\iota$.
\end{theorem}

The dream has always been that one can prove such independence of $\ell$ statements by working with motives instead of $\ell$-adic sheaves. However, in very few instances it has been possible to pull this strategy through, as one usually runs into the issue that motives remain inexplicit. Still, it turns out to work in this case.

We briefly recall the construction of the map $\pi\mapsto \varphi_\pi$ from \cite{FarguesScholze}. We studied the geometry of the stack $\mathrm{Bun}_G$ of $G$-bundles on the Fargues--Fontaine curve associated to $E$. This is the functor that takes any perfectoid space $S$ over $k=\overline{\mathbb F}_q$ to the groupoid of $G$-bundles on the relative Fargues--Fontaine curve $X_S = X_{S,E}$. For any auxiliary prime $\ell\neq p$, we defined a stable $\infty$-category $\mathcal D(\mathrm{Bun}_G,\mathbb Z_\ell)$ of $\ell$-adic sheaves on $\mathrm{Bun}_G$. The stack $\mathrm{Bun}_G$ admits a Harder--Narasimhan stratification indexed by the Kottwitz set $B(G)$, and for any $b\in B(G)$, the locally closed stratum $\mathrm{Bun}_G^b\subset \mathrm{Bun}_G$ has
\[
\mathcal D(\mathrm{Bun}_G^b,\mathbb Z_\ell)\cong \mathcal D(G_b(E),\mathbb Z_\ell)
\]
equivalent to the derived $\infty$-category of smooth representations of $G_b(E)$ on $\mathbb Z_\ell$-modules. Thus, $\mathcal D(\mathrm{Bun}_G,\mathbb Z_\ell)$ is obtained by gluing together, in an infinite semi-orthogonal decomposition, categories of smooth representations of various $p$-adic groups $G_b(E)$. In particular, for $b=1$, we get the category of smooth representations of $G(E)$, that is our main interest.

Via the geometric Satake equivalence, we get Hecke operators acting on $\mathrm{Bun}_G$. For simplicity of exposition, assume that $G$ is split in this introduction. Then for any finite set $I$ and $\mathbb Z_\ell[\sqrt{p}]$-representation $V$ of $\widehat{G}^I$, these yield a functor
\[
\mathcal D(\mathrm{Bun}_G,\mathbb Z_\ell[\sqrt{p}])\to \mathcal D(\mathrm{Bun}_G,\mathbb Z_\ell[\sqrt{p}])^{\ast/W_E^I}
\]
towards (continuously) $W_E^I$-equivariant objects. Here $W_E$ denotes the absolute Weil group of $E$. From this data, one can follow V.~Lafforgue \cite{VLafforgueExcursion} to construct a map from the algebra of excursion operators towards the center of $\mathcal D(\mathrm{Bun}_G,\mathbb Z_\ell[\sqrt{p}])$, or in a more structured fashion following Nadler--Yun \cite{NadlerYun} construct the spectral action on $\mathcal D(\mathrm{Bun}_G,\mathbb Z_\ell[\sqrt{p}])$, which recovers the previous map by passing to endomorphisms of the unit object. The outcome is a map
\[
\mathcal O(Z^1(W_E,\widehat{G}_{\mathbb Q_\ell[\sqrt{p}]})/\widehat{G}_{\mathbb Q_\ell[\sqrt{p}]})\to \mathcal{Z}(G(E),\mathbb Q_\ell[\sqrt{p}])
\]
where $Z^1(W_E,\widehat{G}_{\mathbb Q_\ell[\sqrt{p}]})$ is the scheme of condensed $1$-cocycles $W_E\to \widehat{G}$ over the base $\mathbb Q_\ell[\sqrt{p}]$, and $\mathcal Z(G(E),\mathbb Q_\ell[\sqrt{p}])$ denotes the Bernstein center of the category of $\mathbb Q_\ell[\sqrt{p}]$-linear $G(E)$-representations.

The algebra of excursion operators is known to be independent of $\ell$ by \cite{DHKM}; more precisely, the coarse moduli space
\[
Z^1(W_E,\widehat{G}_{\mathbb Q_\ell[\sqrt{p}]})/\! /\widehat{G}_{\mathbb Q_\ell[\sqrt{p}]}
\]
is defined over $\mathbb Q[\sqrt{p}]$ and in fact agrees with the coarse quotient
\[
Z^1(W_E,\widehat{G}_{\mathbb Q[\sqrt{p}]})/\! /\widehat{G}_{\mathbb Q[\sqrt{p}]}
\]
(where this time the cocycles must be locally constant). Similarly, the Bernstein center of the category of smooth representations of $G(E)$ is defined over $\mathbb Q[\sqrt{p}]$. In \cite[Conjecture I.9.5]{FarguesScholze} we conjectured that the map constructed above is independent of $\ell$, i.e.~for each $\ell$ restricts to a map
\[
\mathcal O(Z^1(W_E,\widehat{G}_{\mathbb Q[\sqrt{p}]})/\widehat{G}_{\mathbb Q[\sqrt{p}]})\to \mathcal{Z}(G(E),\mathbb Q[\sqrt{p}])
\]
that is moreover independent of $\ell$. Our main theorem confirms this.

\begin{theorem}\label{thm:B} For any reductive group $G$ over $E$, \cite[Conjecture I.9.5]{FarguesScholze} on independence of $\ell$ of the map from the excursion algebra to the Bernstein center of $G(E)$ holds true.
\end{theorem}

\begin{remark} As elements of the Bernstein center are uniquely characterized by their action on irreducible smooth representations, one easily sees that this is equivalent to Theorem~\ref{thm:A}.
\end{remark}

Our approach is to repeat the constructions of \cite{FarguesScholze}, replacing $\ell$-adic sheaves with motivic sheaves. This uses the six-functor formalism for rigid-analytic motives developed by Ayoub--Gallauer--Vezzani \cite{AyoubGallauerVezzani}. We will actually rather make use of the slight variant of Berkovich motives \cite{BerkovichMotives}.

The key reason that we can obtain these independence of $\ell$ results using the motivic formalism is that the category $\mathcal D_{\mathrm{mot}}(\mathrm{Div}^1)$ of motivic sheaves on $\mathrm{Div}^1$ can be made completely explicit, relatively to $\mathcal D_{\mathrm{mot}}(k)$, where $k=\overline{\mathbb F}_q$. This is a version of the results of Ayoub and Binda--Gallauer--Vezzani \cite{BindaGallauerVezzani}, reproved in \cite{BerkovichMotives}, concerning $\mathcal D_{\mathrm{mot}}(C)$ for a completed algebraic closure $C$ of $E$. Briefly, $\mathcal D_{\mathrm{mot}}(\mathrm{Div}^1)$ are representations of a suitable version of the Weil--Deligne group, with coefficients in $\mathcal D_{\mathrm{mot}}(k)$. The Weil--Deligne group thus appears here as a relative version of a motivic Galois group. As in the $\ell$-adic case, we use here the strange feature of $\mathrm{Div}^1$ that while it is an incarnation of the curve on which one does geometric Langlands, it is simultaneously covered by one geometric point $\mathrm{Spd}(C)$.

Let us briefly comment on the two most subtle points in adapting everything to motivic sheaves. On the geometric side, the most subtle part is the geometric Satake equivalence. But this has been previously adapted to motivic sheaves by Richarz--Scholbach \cite{RicharzScholbach}, Cass--van den Hove--Scholbach \cite{CassvandenHoveScholbach} and van den Hove \cite{vandenHoveWitt}, at least in the Witt vector affine Grassmannian. There are no problems in adapting their work to the $B_{\mathrm{dR}}^+$-affine Grassmannian required for the Fargues--Fontaine curve. In fact, as we work with \'etale motives, certain arguments are even easier, such as the preservation of mixed Tate sheaves under constant terms.

On the spectral side, there is the issue that the moduli spaces of $L$-parameters used in \cite{DHKM} are canonically defined over $\mathbb Z_\ell$, but their definition over $\mathbb Z[\tfrac 1p]$ relies on certain auxiliary choices (invisible on the level of coarse moduli spaces). We observe here that the motivic formalism actually gives a clear explanation of what is going on. Namely, both on the geometric and on the spectral side, all objects will naturally be linear over $\mathcal D_{\mathrm{mot}}(k)$. This contains $\mathcal D(\mathbb Z[\tfrac 1p])$ fully faithfully, and realizes to $\mathcal D(\mathbb Z_\ell)$ for all $\ell\neq p$, but is much bigger than $\mathcal D(\mathbb Z[\tfrac 1p])$. A priori one runs into the issue that $\mathcal D_{\mathrm{mot}}(k)$ is an unknown category but fortunately it will suffice to understand only a small part of it. Namely, there is the full subcategory
\[
\mathcal D_{MT}(k)\subset \mathcal D_{\mathrm{mot}}(k)
\]
of mixed Tate objects, and the structures most relevant to us are defined over this subcategory. The mixed Tate subcategory has a motivic $t$-structure, and can be explicitly understood as the derived $\infty$-category of quasicoherent sheaves on a certain algebraic stack $\mathrm{MG}_k$ over $\mathbb Z[\tfrac 1p]$, parametrizing line bundles $L$ together with isomorphisms $L/^{\mathbb L} n\cong \mu_n(k)\otimes_{\mathbb Z} R$ for $n$ prime to $p$.\footnote{The letters $\mathrm{MG}$ here stand for ``motivic Galois''.} We construct a canonical Weil--Deligne gerbe $\mathrm{MG}_{\mathrm{Div}^1}$ over $\mathrm{MG}_k$. This is banded by a suitable version of the Weil--Deligne group. The moduli space of $L$-parameters can then be constructed as a canonical algebraic stack over $\mathrm{MG}_k$, parametrizing $\widehat{G}$-bundles on $\mathrm{MG}_{\mathrm{Div}^1}$. After base change along $\mathrm{Spec}(\mathbb Z_\ell)\to \mathrm{MG}_k$, this recovers the previous spaces of $\ell$-adic $L$-parameters. Moreover, there are $\mathbb Z[\tfrac 1p]$-points of $\mathrm{MG}_k$, and the corresponding base change yields algebraic stacks over $\mathbb Z[\tfrac 1p]$ which are the ones constructed by \cite{DHKM}.

{\bf Acknowledgments.} The idea of using rigid-analytic motives to study independence of $\ell$ of $L$-parameters was found while giving the lecture course \cite{ScholzeSixFunctors} on $6$-functor formalisms in the winter term 2022/23. This was of course inspired by the work of Richarz and Scholbach on motivic geometric Satake, and their work in the direction of a motivic version of the results of V.~Lafforgue. I heartily thank David Hansen and the referees for very helpful feedback. The results of this paper were the basis for an ARGOS seminar in the winter term 2024/25, and I heartily thank all the speakers and participants for their feedback. I apologize for the terseness of the presentation.

\section{Basic results}

As discussed in \cite[Section 12]{BerkovichMotives}, the setting used for Berkovich motivic sheaves is slightly incompatible with the setting of small v-stacks. In this paper, we will proceed by maintaining the same geometric objects as in \cite{FarguesScholze} (so in particular, work with v-stacks on the category of perfectoid spaces over $k=\overline{\mathbb F}_q$), and use the pullback functor
\[
a^{\prime\ast}: \mathrm{vStack}\to \mathrm{arcStack}'
\]
taking any $\mathrm{Spa}(R,R^+)$ to $\mathcal M_{\mathrm{arc}}(R)$ (in the setting of analytic Banach rings without fixed norm) in order to use the $6$-functor formalism
\[
X\mapsto \mathcal D_{\mathrm{mot}}(a^{\prime\ast} X).
\]
In particular, we force our sheaf theory to be overconvergent: It takes the same value on $\Spa(R,R^+)$ and $\Spa(R,R^\circ)$. To emphasize this, we will sometimes use the notation
\[
\mathcal D_{\mathrm{mot}}^{\mathrm{oc}}(X) := \mathcal D_{\mathrm{mot}}(a^{\prime\ast} X),
\]
with the superscript $\mathrm{oc}$ indicating that this should be regarded as only the overconvergent part of some $\mathcal D_{\mathrm{mot}}(X)$ that could be defined by using the work of Ayoub--Gallauer--Vezzani \cite{AyoubGallauerVezzani}. As our applications require only overconvergent sheaves, and the category of overconvergent sheaves is in some ways nicer (such as being rigid dualizable on spatial diamonds of finite cohomological dimension), we will stick with this version.

\begin{remark} Another practical reason for this choice is that this overconvergent theory of motivic sheaves admits a realization, for $\ell=p$, into the Hyodo--Kato cohomology constructed in \cite{DeRhamFarguesFontaine};\footnote{This follows from upcoming work of Aoki.} this would not be true of the theory of Ayoub--Gallauer--Vezzani. One could also run the construction of $L$-parameters in \cite{FarguesScholze} for $\ell=p$ using Hyodo--Kato cohomology, and via the Hyodo--Kato realization of motives, our main theorem also applies for $\ell=p$.
\end{remark}

We note that if $n$ is prime to $p$, then by \cite[Proposition 12.3]{BerkovichMotives}
\[
\mathcal D_{\mathrm{mot}}^{\mathrm{oc}}(X)\otimes_{\mathcal D(\mathbb Z[\tfrac 1p])} \mathcal D(\mathbb Z/n)\cong \mathcal D_{\mathrm{et}}^{\mathrm{oc}}(X,\mathbb Z/n)\hookrightarrow \mathcal D_{\mathrm{et}}(X,\mathbb Z/n)
\]
and moreover this equivalence (and embedding) is compatible with tensor and pullback, and hence also with pushforward along proper maps representable in spatial diamonds of finite cohomological dimension (as pullback and proper pushforward preserve overconvergent sheaves). As the motivic $6$-functor formalism is formally extended from this class of maps via the mechanism of \cite[Theorem 5.19]{ScholzeSixFunctors} (or \cite[Theorem 3.4.11]{HeyerMann}), it follows that all $!$-able maps in the motivic $6$-functor formalism are also $!$-able in the (similarly extended) \'etale $6$-functor formalism, and the displayed equivalence and embedding yield maps of $6$-functor formalisms (meaning a transformation of lax symmetric monoidal transformations, i.e.~compatibility with tensor, pullback, and lower-$!$, as also used in \cite[Proposition 4.2.1]{HeyerMann}).\footnote{Note that the construction of \cite[Theorem 5.19]{ScholzeSixFunctors} is functorial in maps of six-functor formalisms. This can be applied to the map
\[
(C,D:\mathrm{Corr}(C,E)\to \mathrm{Pr}^L)\to (C',D':\mathrm{Corr}(C',E')\to \mathrm{Pr}^L)
\]
where $C=C'$ is the $\infty$-category of small v-stacks, $E$ is the class of proper maps representable in spatial diamonds of finite dimtrg, $E'$ is the class of compactifiable maps representable in spatial diamonds of finite dimtrg, $D=\mathcal D_{\mathrm{mot}}^{\mathrm{oc}}$ and $D'=\mathcal D_{\mathrm{et}}(-,\mathbb Z/\ell^n)$. Indeed, this functoriality is actually a consequence of \cite[Theorem 5.19]{ScholzeSixFunctors}, by applying it to the category $\tilde{C}$ which is the total space of the coCartesian fibration classifying the functor $\Delta^1\to \mathrm{Cat}_\infty$ given by $C\to C'$ (so on objects, $\tilde{C}$ is the disjoint union of $C$ and $C'$), with $\tilde{E}=E\sqcup E'$.} Note that \cite{FarguesScholze} and \cite{ECoD} were written before these abstract extensions for $6$-functor formalisms to stacks were available, but whatever definition of $!$-functors was made there by hand agrees with the outcome of those abstract machines (and in any case all arguments in those papers certainly also work when one directly uses the $!$-functors coming out of the abstract machine). In fact, \cite{GulottaHansenWeinstein} shows that the construction of \cite{ECoD} actually yields a six-functor formalism in the sense of Liu--Zheng and Mann, and then the comparison can be deduced from the uniqueness theorem of Dauser--Kuijper \cite{DauserKuijper}. In fact, extending the \'etale $!$-functors to stacky maps in this setting was what led to many of these developments on six-functor formalisms, and was first achieved by Gulotta--Hansen--Weinstein \cite{GulottaHansenWeinstein} based on the work of Liu--Zheng \cite{LiuZhengArtin}, and then more and more streamlined, by Mann \cite[Appendix A.5]{MannThesis}, Heyer--Mann \cite[Theorem 3.4.11]{HeyerMann}, etc.

We will now go through the results of \cite{FarguesScholze} that are about $\ell$-adic sheaves and will see how to prove their motivic versions. This is essentially routine, but one thing to take care of is that quasicompact open immersions do not stay open immersions upon passing to arc-stacks, and are not cohomologically smooth in the $\mathcal D_{\mathrm{mot}}^{\mathrm{oc}}$-formalism.

The first result is \cite[Proposition II.1.21]{FarguesScholze}, whose analogue here is the following. By cohomological smoothness, we always mean in the $\mathcal D_{\mathrm{mot}}^{\mathrm{oc}}$-formalism.

\begin{proposition}\label{prop:div1propersmooth} The map $\Div^1\to \ast$ is proper of finite cohomological dimension, and cohomologically smooth.

Moreover, for any $p$-adic Lie group $H$, the map $\ast/\underline{H}\to \ast$ is cohomologically smooth.
\end{proposition}

\begin{proof} We have $\Div^1=\Spd(\breve{E})/\phi^{\mathbb{Z}}$. The properness was already proved in \cite{FarguesScholze}. For the rest, we can argue on the cover $\Spd(\breve{E})\to \Div^1$, so we have to see that $\Spd(\breve{E})\to \ast$ is $!$-able and cohomologically smooth. If $E$ has equal characteristic, then this is after base change to a geometric point given by a punctured open unit disc, which is indeed $!$-able and cohomologically smooth. In the mixed characteristic case, we can choose a perfectoid $\mathbb Z_p$-extension $\breve{E}_\infty/\breve{E}$, and then consider the composite
\[
\Spd(\breve{E})\to \ast/\underline{\mathbb Z_p}\to \ast.
\]
The first map is locally given by $\Spd(\breve{E}_\infty)\to \ast$ where $\Spd(\breve{E}_\infty)\to \ast$ reduces to the equal characteristic case already considered, via tilting. It remains to see that $\ast/\underline{\mathbb Z_p}\to \ast$ is also cohomologically smooth, which is a special case of the assertion for any $p$-adic Lie group.

This situation of classifying stacks is discussed in great detail by Heyer--Mann \cite[Section 5]{HeyerMann}, see in particular Example 5.3.21 for the case relevant here. (Recall that $\mathcal D_{\mathrm{mot}}$ is $\mathbb Z[\tfrac 1p]$-linear.) Also note that there is a map of $6$-functor formalisms from condensed anima with the sheaf theory considered there, towards arc-stacks with the $\mathcal D_{\mathrm{mot}}$-formalism here, induced from the functor taking any profinite set $S$ to $\underline{S}$, and the symmetric monoidal functor $\mathcal D(S,\mathbb Z[\tfrac 1p])\to \mathcal D_{\mathrm{mot}}^{\mathrm{oc}}(\underline{S},\mathbb Z[\tfrac 1p])$ arising as the pullback of sheaves on the profinite set $S$ to arc-sheaves over $\underline{S}$. Moreover, cohomological smoothness is preserved under maps of $6$-functor formalisms (where maps are required to commute with pullback, tensor, and lower-$!$, but not their right adjoints), by the diagrammatic characterization \cite[Definition 4.5.1]{HeyerMann} of cohomological smoothness in terms of being suave with invertible suave dualizing complex, and the functoriality \cite[Proposition 4.2.1]{HeyerMann} of the $2$-category of kernels (so that adjoints, and hence suave objects and suave duals, are preserved).
\end{proof}

The next result is \cite[Proposition II.2.5]{FarguesScholze}.

\begin{proposition}\label{prop:BCsmooth} Let $\lambda\in \mathbb Q$.
\begin{enumerate}
\item[{\rm (i)}] If $\lambda>0$, then $\BC(\mathcal O(\lambda))\to \ast$ is cohomologically smooth.
\item[{\rm (ii)}] If $\lambda<0$, then $\BC(\mathcal O(\lambda)[1])\to \ast$ is cohomologically smooth.
\end{enumerate}
\end{proposition}

\begin{proof} Replacing $E$ by a finite unramified extension, we can assume that $\lambda=n$ is an integer. For $n=1$, we know that $\BC(\mathcal O(1))\to \ast$ is an open unit disc, which is cohomologically smooth. In general, we can base change to some geometric point $\Spa(C)$, pick an untilt $C^\sharp$ over $E$, and use an exact sequence
\[
0\to \BC(\mathcal O(n))\to \BC(\mathcal O(n+1))\to (\mathbb A^1_{C^\sharp})^\diamond\to 0
\]
to argue by induction, using cohomological smoothness of $\mathbb A^1$.

For negative $n$, we start from the similar exact sequence for $n=-1$ which yields
\[
\BC(\mathcal O(-1)[1]) = (\mathbb A^1_{C^\sharp})^\diamond/\underline{E}.
\]
In particular, we can consider the composite
\[
\BC(\mathcal O(-1)[1])\to \ast/\underline{E}\to \ast
\]
where both maps are known to be cohomologically smooth. Finally, induct to more negative $n$ by using the exact sequences
\[
0\to (\mathbb A^1_{C^\sharp})^\diamond\to \BC(\mathcal O(n-1)[1])\to \BC(\mathcal O(n)[1])\to 0.\qedhere
\]
\end{proof}

More generally, we have \cite[Proposition II.3.5]{FarguesScholze}.

\begin{proposition}\label{prop:BCsmoothgeneral} Let $S$ be a perfectoid space over $\mathbb F_q$ and let $[\mathcal E_1\to \mathcal E_0]$ be a map of vector bundles on $X_S$ such that at all geometric points of $S$, all Harder--Narasimhan slopes of $\mathcal E_1$ are negative and all Harder--Narasimhan slopes of $\mathcal E_0$ are positive. Then
\[
\BC([\mathcal E_1\to \mathcal E_0])\to S
\]
is cohomologically smooth.
\end{proposition}

\begin{proof} The proof of \cite[Proposition II.3.5 (iii)]{FarguesScholze} applies without change, once one has the analogue of \cite[Proposition 23.13]{ECoD} which says that being cohomologically smooth is cohomologically smooth-local. This is a general fact of $6$-functor formalisms, see \cite[Lemma 4.5.8 (i)]{HeyerMann}.
\end{proof}

Another situation is \cite[Proposition II.3.7]{FarguesScholze} concerning projectivized Banach--Colmez spaces. This follows by using the factorization over $\ast/\underline{E}^\times$ as before.

In \cite[Chapter III]{FarguesScholze}, the only relevant assertion is the cohomological smoothness of the connected components $\tilde{\mathcal G}_b^\circ$; but those are successive extensions of positive Banach--Colmez spaces.

More relevant is \cite[Chapter IV]{FarguesScholze}. First, the meaning of Artin v-stack changes now, as ``cohomological smoothness'' must be taken in the sense of the $\mathcal D_{\mathrm{mot}}^{\mathrm{oc}}$-formalism; we thus use the following notion.

\begin{definition} An Artin arc-stack is a v-stack $X$ such that the diagonal $\Delta_X: X\to X\times X$ is representable in partially proper locally spatial diamonds, and there is a separated and $\mathcal D_{\mathrm{mot}}^{\mathrm{oc}}$-cohomologically smooth surjection $Y\to X$ from a partially proper locally spatial diamond $Y$.
\end{definition}

Most importantly, as quasicompact open immersions are not cohomologically smooth, one has to a bit careful. In particular, in \cite[Example IV.1.7, Proposition IV.1.8, Example IV.1.9]{FarguesScholze} one must assume that the locally spatial diamonds (or maps) are also partially proper.

Next, we have \cite[Proposition IV.1.18, Theorem IV.1.19]{FarguesScholze} whose analogue is the following result.

\begin{proposition}\label{prop:BunGArtin} For any $\mu\in X_\ast(T)^+$, the Schubert cell $\mathrm{Gr}_{G,\mu}\to \mathrm{Spd}(E)$ is cohomologically smooth. The Beauville--Laszlo map
\[
\bigsqcup_\mu [\mathrm{Gr}_{G,\mu}/\underline{G(E)}]\to \mathrm{Bun}_G
\]
is a cohomologically smooth surjection. The arc-stack $\mathrm{Bun}_G$ is a cohomologically smooth Artin arc-stack with dualizing complex isomorphic to $\mathbb Z[\tfrac 1p]$. Each Harder--Narasimhan stratum $\mathrm{Bun}_G^b$ is a cohomologically smooth Artin arc-stack of dimension $-\langle 2\rho,\nu_b\rangle$.
\end{proposition}

We note that the dualizing complexes of the individual Harder--Narasimhan strata are nontrivial, and have been computed by Hamann--Imai \cite{HamannImai} (with torsion coefficients, but their arguments adapt to the motivic case).

\begin{proof} The proof of cohomological smoothness of $\mathrm{Gr}_{G,\mu}$ given in \cite[Proposition VI.2.4]{FarguesScholze} also works in the motivic setting, as it reduces everything to smooth algebraic varieties. The Beauville--Laszlo map is locally a product with Schubert cells and thus also cohomologically smooth; and this also lets one deduce that the dualizing complex of $\mathrm{Bun}_G$ must be locally isomorphic to $\mathbb Z[\tfrac 1p]$. On the semistable points, it can be trivialized by choosing a $\mathbb Z[\tfrac 1p]$-valued Haar measure on $G_b(E)$. This trivialization necessarily uniquely extends to all of $\mathrm{Bun}_G$ by purity, using that all other Harder--Narasimhan strata are smooth of negative dimension. This last statement follows from the presentation $\Bun_G^b\cong [\ast/\tilde{\mathcal G}_b]$ where the connected component of the identity is cohomologically smooth (of dimension $\langle 2\rho,\nu_b\rangle$) and the components $G_b(E)$ are a $p$-adic Lie group.
\end{proof}

In \cite[Section IV.2]{FarguesScholze}, there is a general discussion of universal local acyclicity. In the intervening years, this part has become streamlined and generalized to arbitrary $6$-functor formalisms, leading to what is now called $f$-suave sheaves, see \cite[Definition 4.4.1]{HeyerMann}. The original characterization \cite[Definition IV.2.1]{FarguesScholze} of universal local acyclicity is not available for us here (everything is overconvergent, and we do not have quasicompact open immersions), but we will not need it.

While \cite[Section IV.3]{FarguesScholze} is purely geometric, the next \cite[Section IV.4]{FarguesScholze} proves the important Jacobian criterion for cohomological smoothness. This fortunately holds true motivically (assuming that $Z$ is partially proper):

\begin{theorem}\label{thm:jacobiancriterion} Let $S$ be a perfectoid space and let $Z\to X_S$ be a smooth map of sous-perfectoid spaces such that $Z$ admits a Zariski closed immersion into a partially proper open subset of (the adic space) $\mathbb P^n_{X_S}$ for some $n\geq 0$. Then, with $\mathcal M_Z^{\mathrm{sm}}\subset \mathcal M_Z$ defined as in \cite{FarguesScholze}, the map
\[
\mathcal M_Z^{\mathrm{sm}}\to S
\]
is cohomologically smooth.
\end{theorem}

\begin{proof} Most of the proof carries over without change. In particular, the constant sheaf is suave by the same argument as for \cite[Proposition IV.4.27]{FarguesScholze}, using that the proof of \cite[Lemma IV.4.28]{FarguesScholze} yields maps that are cohomologically smooth in the $\mathcal D_{\mathrm{mot}}^{\mathrm{oc}}$-formalism. The only real change is at the end of the proof, after \cite[Lemma IV.4.30]{FarguesScholze}. We write out the argument (copying almost verbatim), using all the notation from there. We make the small change that whenever a localization to a small neighborhood was made in the proof, we assume that this neighborhood is partially proper (instead of quasicompact). This means that in some of the spreading out argument used below, we have to shrink this neighborhood further.

Let $f': \mathcal M_{Z'}\to S'$ be the projection, with fibres $f^{(n)}$ and $f^{(\infty)}$. Both $\mathbb Z[\tfrac 1p]$ and $Rf'^! \mathbb Z[\tfrac 1p]$ are $f'$-suave. In particular, the formation of $Rf'^! \mathbb Z[\tfrac 1p]$ commutes with base change, and we see that the restriction of $Rf'^!\mathbb Z[\tfrac 1p]$ to the fibre over $\infty$ is \'etale locally isomorphic to $\mathbb Z[\tfrac 1p](d)[2d]$, as an open subset of $\BC(s^\ast T_{Z/X_S})$. As $S$ is strictly totally disconnected, one can choose a global isomorphism with $\mathbb Z[\tfrac 1p](d)[2d]$.

The map from $\mathbb Z[\tfrac 1p](d)[2d]$ to the fibre of $Rf'^! \mathbb Z[\tfrac 1p]$ over $\infty$ extends to a small neighborhood, at least after shrinking $\mathcal M_{Z'}$ to a smaller neighborhood of the section $s$: Indeed, the extension would work if $\mathbb Z[\tfrac 1p](d)[2d]\in \mathcal D_{\mathrm{mot}}^{\mathrm{oc}}(\mathcal M_{Z'})$ was compact. It only fails to be compact because $\mathcal M_{Z'}$ fails to be quasicompact, but by shrinking $\mathcal M_{Z'}$ to a smaller neighborhood of $s$, we can correct for this. Thus, we can assume that there is a map
\[
\beta: \mathbb Z[\tfrac 1p](d)[2d]\to Rf'^! \mathbb Z[\tfrac 1p]
\]
that is an isomorphism in the fibre over $\infty$. We can assume that this map is $\gamma$-equivariant (passing to a smaller neighborhood). Let $Q$ be the cone of $\beta$. Then $Q$ is still $f'$-suave, as is its Verdier dual
\[
\mathbb D_{\mathcal M_{Z'}/S'}(Q) = R\sHom_{\mathcal M_{Z'}}(Q,Rf'^! \mathbb Z[\tfrac 1p]).
\]
Choosing two partially proper open neighborhoods $\mathcal M_{Z'''}\subset \mathcal M_{Z''}$ of the section $s$, strictly contained in another, the induced transition map
\[
Rf'''_!\mathbb D_{\mathcal M_{Z'''}/S'}(Q)\to Rf''_!\mathbb D_{\mathcal M_{Z''}/S'}(Q)
\]
is compact, and its restriction to $S\times \{\infty\}$ is zero; more precisely (squeezing another partially proper open neighborhood in between), we can factor the map over a compact object $K$ so that the map $K\to Rf''_!\mathbb D_{\mathcal M_{Z''}/S'}(Q)$ becomes zero after restriction to $S\times \{\infty\}$. This implies that its restriction to $S\times \{n,n+1,\ldots,\infty\}$ is zero for some $n\gg 0$ by compactness of $K$, and in particular so is
\[
Rf'''_!\mathbb D_{\mathcal M_{Z'''}/S'}(Q)\to Rf''_!\mathbb D_{\mathcal M_{Z''}/S'}(Q).
\]
Taking Verdier duals, this implies that also a similar transition map $Rf''_\ast Q\to Rf'''_\ast Q$ is zero.

In particular, for all $n\geq n_0$, the transition map on the fibre over $S\times\{n\}$ is zero. Using the $\gamma$-equivariance, this implies that the transition maps
\[
Rf^{(n)}_\ast (Q|_{\mathcal M_Z^{(n_0)}})|_{\mathcal M_{Z^{(n)}}}\to Rf^{(n+1)}_\ast (Q|_{\mathcal M_Z^{(n_0)}})|_{\mathcal M_{Z^{(n+1)}}}
\]
are zero, regarding $\mathcal M_{Z^{(n)}}\subset \mathcal M_{Z^{(n_0)}}$ as an open subset. Taking the colimit over all $n$ and using that the system $\mathcal M_{Z^{(n)}}\subset \mathcal M_{Z^{(n_0)}}$ has intersection $s(S)\subset \mathcal M_Z$ and is cofinal with a system of spatial diamonds of finite cohomological dimension (as can be checked in the case of projective space), finitaryness of motivic sheaves implies that
\[
s^\ast Q|_{\mathcal M_{Z^{(n_0)}}} = \varinjlim_n Rf^{(n)}_\ast (Q|_{\mathcal M_{Z^{(n_0)}}})|_{\mathcal M_{Z^{(n)}}} = 0
\]
and thus the map
\[
s^\ast \beta|_{\mathcal M_{Z'}}: \mathbb Z[\tfrac 1p](d)[2d]\to s^\ast Rf'^! \mathbb Z[\tfrac 1p]
\]
is an isomorphism, as desired.
\end{proof}

The results on partially compactly supported cohomology in \cite[Section IV.5]{FarguesScholze} mostly carry over without change. In the definitions, one has to change the use of quasicompact open subsets into partially proper open subsets, but there are also enough of those (by taking a countable union of strict inclusions of quasicompact open subsets). A cheap way to see that \cite[Lemma IV.5.1]{FarguesScholze} adapts to the motivic case is to observe that this partially compactly supported cohomology of a punctured open unit disc is mixed Tate and commutes with base change (by obvious excision triangles), so defines an object of $\mathcal D_{MT}(k)$; but on this category, the combined $\ell$-adic realization functors are conservative. Also the results on hyperbolic localization in \cite[Section IV.6]{FarguesScholze} are unchanged --- all the key arguments are of point-set topology nature.\footnote{Alternatively, recent work in progress by Heyer--Mann yields hyperbolic localization results for arbitrary six-functor formalisms, under minimal assumptions.} By contrast, the results on Drinfeld's lemma in \cite[Section IV.7]{FarguesScholze} will be different, as $\mathcal D_{\mathrm{mot}}^{\mathrm{oc}}(\Div^1)$ is not simply equivalent to $W_E$-representations. Still, we have the following result on dualizable mixed Tate motives, as defined in \cite[Definition 10.7]{BerkovichMotives}, where we use the algebraic stack $\mathrm{MG}_{\mathrm{Div}^1}$ introduced in Section~\ref{sec:div1} below. Here and in the following, for any finite set $I$, we define
\[
\mathrm{Div}^I := (\mathrm{Div}^1)^I.
\]

\begin{theorem}\label{thm:drinfeldlemma} For any finite set $I$ and any $\mathbb Z[\tfrac 1p]$-algebra $\Lambda$, the natural functor 
\[
(\mathcal D_{MT}(\mathrm{Div}^1,\Lambda)^{\mathrm{dual}})^{\otimes I/\mathcal D_{MT}(k,\Lambda)^\omega}\to \mathcal D_{MT}(\mathrm{Div}^I,\Lambda)^{\mathrm{dual}}
\]
is an equivalence of stable $\infty$-categories, and this is equivalent to the stable $\infty$-category of perfect complexes on
\[
\mathrm{MG}_{\mathrm{Div}^I,\Lambda} := \mathrm{MG}_{\mathrm{Div}^1}^{I_{/\mathrm{MG}_k}}\times \mathrm{Spec}(\Lambda).
\]
\end{theorem}

Here, we take tensor products of small idempotent complete stable $\infty$-categories; under the equivalence with compactly generated presentable stable $\infty$-categories, this corresponds to the Lurie tensor product there.

\begin{proof} We already know fully faithfulness by \cite[Corollary 10.6]{BerkovichMotives} (using the normed version of the stacks to make them quasicompact and quasiseparated; the resulting pullback functors on sheaf categories are fully faithful (as they are a pullback under a $\mathbb R_{>0}$-torsor) and an equivalence for $k$). With torsion coefficients, the equivalence is \cite[Proposition IV.7.3]{FarguesScholze}, so it suffices to prove the statement rationally. But then we can filter by weights, and one can reduce to the case of weight $0$, in which case we have to classify \'etale $\Lambda$-local systems. But these are equivalent to $W_E^I$-representations, by the proof of \cite[Proposition IV.7.3]{FarguesScholze} (reducing now $\Lambda$ to completions at closed points of finitely generated $\mathbb Z$-algebras, instead of $\mathbb F_\ell$-algebras -- after this reduction). For the first step of the proof of \cite[Proposition IV.7.3]{FarguesScholze}, we note that for any small v-stack $X$ and any $\mathbb Z[\tfrac 1p]$-algebra $\Lambda$, the functor
\[
\mathcal D_{\mathrm{et}}^{\mathrm{oc}}(X,\Lambda)\to \mathcal D_{\mathrm{et}}^{\mathrm{oc}}(X\times \mathrm{Spa}(C),\Lambda)
\]
is fully faithful. If $n\Lambda=0$ for some $n$ prime to $p$, this is \cite[Theorem 19.5]{ECoD}. But the argument given there works in general for overconvergent sheaves; the key computation is that the \'etale $\Lambda$-cohomology of the universal cover of an annulus is trivial. With rational coefficients, even the annulus itself has trivial cohomology.
\end{proof}

\section{$\mathcal D_{\mathrm{mot}}(\Bun_G)$}

Finally, we are in a position to analyze $\mathcal D_{\mathrm{mot}}(\Bun_G) = \mathcal D_{\mathrm{mot}}^{\mathrm{oc}}(\Bun_G)$. As there is no possibility for non-overconvergent sheaves on $\Bun_G$, we leave out the superscript $^{\mathrm{oc}}$ here.

We have the following analogue of \cite[Theorem V.0.1]{FarguesScholze}.

\begin{theorem}\leavevmode
\begin{enumerate}
\item[{\rm (o)}] For any $b\in B(G)$, the chart
\[
\pi_b: \mathcal M_b\to \Bun_G
\]
is cohomologically smooth, and $\mathcal M_b$ is cohomologically smooth over $\ast/\underline{G_b(E)}$.
\item[{\rm (i)}] Via excision triangles, there is an infinite semiorthogonal decomposition of $\mathcal D_{\mathrm{mot}}(\Bun_G,\Lambda)$ into the various $\mathcal D_{\mathrm{mot}}(\Bun_G^b)$ for $b\in B(G)$.
\item[{\rm (ii)}] For each $b\in B(G)$, pullback along
\[
\Bun_G^b\cong [\ast/\tilde{G}_b]\to [\ast/\underline{G_b(E)}]
\]
gives an equivalence
\[
\mathcal D_{\mathrm{mot}}(\ast/\underline{G_b(E)})\cong \mathcal D_{\mathrm{mot}}(\Bun_G^b),
\]
and
\[
\mathcal D_{\mathrm{mot}}(\ast/\underline{G_b(E)})\cong \mathcal D_{\mathrm{mot}}(\ast)\otimes_{\mathcal D(\mathbb Z[\tfrac 1p])} \mathcal D(G_b(E),\mathbb Z[\tfrac 1p])
\]
is equivalent to the stable $\infty$-category of smooth representations of $G_b(E)$ with coefficients in
\[
\mathcal D_{\mathrm{mot}}(\ast) = \mathcal D_{\mathrm{mot}}(k).
\]
\item[{\rm (iii)}] The category $\mathcal D_{\mathrm{mot}}(\Bun_G)$ is compactly generated, and a complex $A\in \mathcal D_{\mathrm{mot}}(\Bun_G)$ is compact if and only if for all $b\in B(G)$, the restriction
\[
i^{b\ast} A\in \mathcal D_{\mathrm{mot}}(\Bun_G^b)\cong \mathcal D_{\mathrm{mot}}(\ast)\otimes_{\mathcal D(\mathbb Z[\tfrac 1p])} \mathcal D(G_b(E),\mathbb Z[\tfrac 1p])
\]
is compact, and zero for almost all $b$. Here, compactness in $\mathcal D_{\mathrm{mot}}(\ast)\otimes_{\mathcal D(\mathbb Z[\tfrac 1p])} \mathcal D(G_b(E),\mathbb Z[\tfrac 1p])$ is equivalent to lying in the thick triangulated subcategory generated by $c\text-\mathrm{Ind}_K^{G_b(E)} M$ as $K$ runs over open pro-$p$-subgroups of $G_b(E)$ and $M\in \mathcal D_{\mathrm{mot}}(\ast)$ is compact.
\item[{\rm (iv)}] On the subcategory $\mathcal D_{\mathrm{mot}}(\Bun_G)^\omega\subset \mathcal D_{\mathrm{mot}}(\Bun_G)$ of compact objects, there is a Bernstein--Zelevinsky duality functor
\[
\mathbb D_{BZ}: (\mathcal D_{\mathrm{mot}}(\Bun_G)^\omega)^{\mathrm{op}}\to \mathcal D_{\mathrm{mot}}(\Bun_G)^\omega
\]
with a functorial identification
\[
R\Hom(A,B)\cong \pi_!(\mathbb D_{BZ}(A)\dotimes_\Lambda B)
\]
for $B\in \mathcal D_{\mathrm{mot}}(\Bun_G)$, where $\pi: \Bun_G\to \ast$ is the projection. The functor $\mathbb D_{BZ}$ is an equivalence, and $\mathbb D_{BZ}^2$ is naturally equivalent to the identity. It is compatible with usual Bernstein--Zelevinsky duality on $\mathcal D(G_b(E),\mathbb Z[\tfrac 1p])$ for basic $b\in B(G)$.
\item[{\rm (v)}] An object $A\in \mathcal D_{\mathrm{mot}}(\Bun_G)$ is suave (with respect to $\Bun_G\to \ast$) if and only if for all $b\in B(G)$, the restriction
\[
i^{b\ast} A\in \mathcal D_{\mathrm{mot}}(\Bun_G^b)\cong \mathcal D_{\mathrm{mot}}(\ast)\otimes_{\mathcal D(\mathbb Z[\tfrac 1p])} \mathcal D(G_b(E),\mathbb Z[\tfrac 1p])
\]
is admissible, i.e.~for all pro-$p$ open subgroups $K\subset G_b(E)$, the invariants $(i^{b\ast} A)^K\in \mathcal D_{\mathrm{mot}}(\ast)$ are compact. Suave objects are preserved by Verdier duality, and satisfy Verdier biduality.
\end{enumerate}
\end{theorem}

Between the writing of \cite{FarguesScholze} and now, Mann found the dual notion ``prim'' of suave (cf.~\cite[Definition 4.4.1]{HeyerMann}), and Hansen--Mann observed that the compact objects considered in (iv) are also the prim objects, and Bernstein--Zelevinsky duality is prim duality.

\begin{proof} All the proofs immediately adapt. Part (o) follows as before from the Jacobian criterion. Part (i) follows from excision, which holds for motivic sheaves. For part (ii), one uses that pullback to gerbes banded by (extensions of) positive Banach--Colmez spaces is an equivalence (using that positive Banach--Colmez spaces are cohomologically smooth and contractible, cf.~\cite[Proposition V.2.1]{FarguesScholze}). To prove the equivalence
\[
\mathcal D_{\mathrm{mot}}(\ast)\otimes_{\mathcal D(\mathbb Z[\tfrac 1p])} \mathcal D(G_b(E),\mathbb Z[\tfrac 1p])\cong \mathcal D_{\mathrm{mot}}(\ast/\underline{G_b(E)}),
\]
one can follow the proof of \cite[Theorem V.1.1]{FarguesScholze}. Alternatively, one can use the map of six-functor formalisms from condensed anima to arc-stacks used in the proof of Proposition~\ref{prop:div1propersmooth}, and \cite[Proposition 5.1.12]{HeyerMann}, to construct a functor
\[
\mathcal D_{\mathrm{mot}}(\ast)\otimes_{\mathcal D(\mathbb Z[\tfrac 1p])} \mathcal D(G_b(E),\mathbb Z[\tfrac 1p])\to \mathcal D_{\mathrm{mot}}(\ast/\underline{G_b(E)}).
\]
To show that this is an equivalence, one first reduces to the similar assertion for $\ast/\underline{K}$ for some open pro-$p$-subgroup $K\subset G_b(E)$, via the \'etale cover $\ast/\underline{K}\to \ast/\underline{G_b(E)}$ (and the resulting monadic adjunction which is preserved by Lurie tensor products). But then $\ast\to \ast/\underline{K}$ is proper, and the resulting comonadic adjunction is preserved by Lurie tensor products.

In (iii), the description of the compact objects in these representation categories with coefficients in $\mathcal D_{\mathrm{mot}}(\ast)$ follows from the general description of compact objects in Lurie tensor products of compactly generated categories, as generated by tensor products of compact objects in the factors. The charts from (o) yield left adjoint functors to the pullback functors $i_b^\ast$, by the analogue of \cite[Proposition V.4.2]{FarguesScholze}; these left adjoints then preserve compact objects, and yield a generating class of compact objects. To see that their restrictions to all strata are compact, it suffices to see that $i_{b!}$ also preserves compact objects. This follows from the finite cohomological dimension of the spatial diamond $\widetilde{\mathcal M}^\circ_b/\underline{K}$ as in the proof of \cite[Theorem V.4.1]{FarguesScholze}.

In part (iv), we check that the prim objects for $\mathrm{Bun}_G\to \ast$ are precisely the compact objects; then prim duality has the desired properties, cf.~\cite[Proposition 6.8]{ScholzeSixFunctors}. First, by \cite[Proposition 6.16]{ScholzeSixFunctors}, any prim object must be compact. For the converse, it suffices to see that the generating objects from (iii) are all prim. More precisely, the objects $A_K^b$ arising as left adjoints of $i_b^\ast$ evaluated on $c\text-\mathrm{Ind}_K^{G_b(E)} \Lambda$, for open pro-$p$-subgroups $K\subset G_b(E)$, have as prim dual $i_{b!} c\text-\mathrm{Ind}_K^{G_b(E)} \Lambda$ up to shift and twist. It is easy to construct the required unit and counit of the adjunction, using the left adjoint property, and check that the composites are the identity.

Finally, for part (v), one argues as in \cite[Section V.7]{FarguesScholze}, using the next proposition applied with $S$ also given by $\mathrm{Bun}_G$.
\end{proof}

Let us also note the following result whose analogue for $\mathcal D_{\mathrm{et}}$ was not explicitly stated in \cite{FarguesScholze}.

\begin{proposition}\label{prop:DmotBunGbasechange} For any small arc-stack $S$, the exterior tensor product gives an equivalence
\[
\mathcal D_{\mathrm{mot}}^{\mathrm{oc}}(S)\otimes_{\mathcal D_{\mathrm{mot}}(\ast)} \mathcal D_{\mathrm{mot}}(\Bun_G)\cong \mathcal D_{\mathrm{mot}}^{\mathrm{oc}}(\Bun_G\times S).
\]
\end{proposition}

\begin{proof} This is easy to see stratum by stratum. To see that the gluing functors are compatible, use the description of the left adjoints to $i^{b\ast}$ in terms of the charts $\mathcal M_b$; this description also holds after any base change. To see the latter, we need to see the analogue of \cite[Proposition V.4.2]{FarguesScholze}: for any $S$ and any $A\in \mathcal D_{\mathrm{mot}}^{\mathrm{oc}}(S\times \tilde{\mathcal M}_b)$ with pullback $A_0 = i^\ast A\in \mathcal D_{\mathrm{mot}}^{\mathrm{oc}}(S)$ along $i: S\hookrightarrow S\times \tilde{\mathcal M}_b$, the map
\[
R\Gamma(S\times \tilde{\mathcal M}_b,A)\to R\Gamma(S,A_0)
\]
is an isomorphism. By descent, we can assume that $S$ is strictly totally disconnected. Then it amounts to the vanishing of the partially compactly supported cohomology of the partially proper locally spatial diamond $S\times \tilde{\mathcal M}_b^\circ$.
\end{proof}

\section{$\mathcal D_{\mathrm{mot}}(\mathrm{Div}^1)$}\label{sec:div1}

 Recall the diamond $\mathrm{Div}^1 = \mathrm{Spd}(\breve{E})/\phi$. The goal of this section is to give an explicit description of $\mathcal D_{\mathrm{mot}}(\mathrm{Div}^1)$ in terms of $\mathcal D_{\mathrm{mot}}(k)$, where $k=\overline{\mathbb F}_q$.

Recall the algebraic stack $\mathrm{MG}_k$ over $\mathbb Z[\tfrac 1p]$ parametrizing, over an (animated) ring $R$, an invertible $R$-module $L$ together with compatible $R/^{\mathbb L} n$-linear isomorphisms $L/^{\mathbb L} n\cong \mu_n(k)\otimes^{\mathbb L} R$. This comes with a distinguished equivalence
\[
\mathcal D_{\mathrm{qc}}(\mathrm{MG}_k)\cong \mathcal D_{MT}(k),
\]
sending the universal $L$ to $\mathbb Z(1)$.

As preparation for the following definition, recall that for any field $K$, the classifying stack $\ast/\mathrm{Gal}(\overline{K}/K)$ is canonical. One way to see this is to write down its functor of points. Most naturally, $\ast/\mathrm{Gal}(\overline{K}/K)$ should be considered here as a stack on the category of profinite sets, i.e.~a condensed anima. In this world, it is the ``moduli space of algebraic closures of $K$'' in the sense that for all profinite sets $S$, maps $S\to \ast/\mathrm{Gal}(\overline{K}/K)$ are given by a sheaf of $K$-algebras $\tilde{K}_S$ over $S$ such that at each point $s\in S$, the stalk $\tilde{K}_s$ is an algebraic closure of $K$. As sheaves on $S$ are equivalent to modules over $\mathrm{Cont}(S,\mathbb Z)$, this datum is also equivalent to the datum of a $K$-algebra $\tilde{K}$ with $\pi_0 \mathrm{Spec}(\tilde{K})\cong S$ such that each connected component is the spectrum of an algebraic closure of $K$.

We will however have to transmute this condensed anima back to the world of algebraic stacks, via sending a profinite set $S$ to $\mathrm{Spec}(S,\mathbb Z)$. One can then describe its functor of points on suitably local spaces, for example w-strictly local rings $R$, as being a $K$-algebra $\tilde{K}$ with $\pi_0 \mathrm{Spec}(\tilde{K})\cong \pi_0 \mathrm{Spec}(R)$, such that each connected component is an algebraic closure of $K$.

We now give an analogue of this definition for $K=\breve{E}$, and replacing the algebraic closure by its completion. Moreover, we replace the absolute Galois group by a motivic version.

\begin{definition} The stack $\mathrm{MG}_{\breve{E}}$ over $\mathrm{MG}_k$ parametrizes, on w-strictly local rings $R$, a strictly totally disconnected perfectoid $\breve{E}$-algebra $C$ with an isomorphism of profinite sets $\pi_0 \mathrm{Spa}(C)\cong \pi_0 \mathrm{Spec}(R)$ and each completed residue field a completed algebraic closure of $\breve{E}$; together with a map
\[
C^\times/(1+C_{<1})[-1]\to L
\]
of sheaves on $\pi_0 \mathrm{Spa}(C)\cong \pi_0\mathrm{Spec}(R)$, compatibly lifting the $\mathbb Z/n$-linear map
\[
(C^\times/(1+C_{<1})[-1])/^{\mathbb L} n = \mu_n(k)\to L/^{\mathbb L} n
\]
for all $n$ prime to $p$. (Here, abusively, $\mu_n(k)$ denotes the constant sheaf on $\pi_0 \mathrm{Spec}(R)$.)

On $\mathrm{MG}_{\breve{E}}$, there is a Frobenius action, acting on $\breve{E}$ and simultaneously scaling the map $C^\times/(1+C_{<1})[-1]\to L$ by $q$. Let $\mathrm{MG}_{\Div^1}$ be the quotient of $\mathrm{MG}_{\breve{E}}$ by this Frobenius action.
\end{definition}

If one fixes a separable closure $\overline{\breve{E}}$ of $\breve{E}$, then there is a map from $\mathrm{MG}_{\breve{E}}$ to the classifying space of $I_E=\mathrm{Gal}(\overline{\breve{E}}/\breve{E})$, parametrizing isomorphisms of $C$ with the completed algebraic closure of $\overline{\breve{E}}$ (base changed to $\pi_0 R$). The fibre of the map
\[
\mathrm{MG}_{\breve{E}}\to \ast/I_E
\]
is precisely the stack $\mathrm{MG}_C$ introduced in \cite{BerkovichMotives}. It follows that the map
\[
\mathrm{MG}_{\breve{E}}\to \mathrm{MG}_k
\]
is a gerbe banded by a group scheme that is an extension
\[
0\to \varprojlim_n L\to \mathrm{ID}_E\to I_E\to 0
\]
where $\mathrm{ID}_E$ stands for ``inertia--Deligne''; it is the ``inertia subgroup of the Weil--Deligne group''. Picking $\overline{\breve{E}}$ together with a uniformizer $\pi$ of $E$ and roots $\pi^{1/n}$ in $\overline{\breve{E}}$ for $n$ prime to $p$ yields a splitting of this gerbe.

Using Artin motives and the explicit description of $\mathbb Z(1)=\overline{\mathbb G}_m[-1]$, there is a natural $\mathcal D_{\mathrm{qc}}(\mathrm{MG}_k)\cong \mathcal D_{MT}(k)$-linear symmetric monoidal functor
\[
\mathcal D_{\mathrm{qc}}(\mathrm{MG}_{\breve{E}})\to \mathcal D_{MT}(\mathrm{Spd}(\breve{E}))
\]
which descends to a $\mathcal D_{\mathrm{qc}}(\mathrm{MG}_k)\cong \mathcal D_{MT}(k)$-linear symmetric monoidal functor
\[
\mathcal D_{\mathrm{qc}}(\mathrm{MG}_{\Div^1})\to \mathcal D_{MT}(\Div^1).
\]
More precisely, to construct the first functor, we aim to construct more generally for any pro-finite \'etale $\breve{E}$-algebra $\breve{E}'$ a similar functor
\[
\mathcal D_{\mathrm{qc}}(\mathrm{MG}_{\breve{E}'})\to \mathcal D_{MT}(\mathrm{Spd}(\breve{E}')),
\]
using the analogue $\mathrm{MG}_{\breve{E}'}$ obtained by replacing $\breve{E}$ by $\breve{E}'$ in the definition. By descent, it suffices to give this construction when the completion $C$ of $\breve{E}'$ is strictly totally disconnected; this descent step implicitly makes use of Artin motives. In the strictly totally disconnected case $\mathrm{MG}_{\breve{E}'}$ reduces to the moduli space of maps $C^\times/(1+C_{<1})[-1]\to L$; but $C^\times/(1+C_{<1})[-1] = \mathbb Z(1)(C)$, so we get the desired functor in this case, cf.~also \cite[Section 11]{BerkovichMotives}.

\begin{theorem} The functor
\[
\mathcal D_{\mathrm{qc}}(\mathrm{MG}_{\Div^1})\to \mathcal D_{MT}(\Div^1)
\]
is an equivalence, and the induced functor
\[
\mathcal D_{\mathrm{qc}}(\mathrm{MG}_{\Div^1})\otimes_{\mathcal D_{\mathrm{qc}}(\mathrm{MG}_k)} \mathcal D_{\mathrm{mot}}(k)\to \mathcal D_{\mathrm{mot}}(\Div^1)
\]
is an equivalence.
\end{theorem}

\begin{proof} It suffices to prove the similar assertions for $\breve{E}$ in place of $\mathrm{Div}^1$. In that case, both sides are naturally hypercomplete sheaves over the \'etale site of $\mathrm{Spec}(\breve{E})$, and it suffices to understand the stalk at the algebraic closure. This reduces us to the assertions
\[
\mathcal D_{\mathrm{qc}}(\mathrm{MG}_C)\cong \mathcal D_{MT}(\mathrm{Spa}(C)),
\]
\[
\mathcal D_{\mathrm{qc}}(\mathrm{MG}_C)\otimes_{\mathcal D_{\mathrm{qc}}(\mathrm{MG}_k)} \mathcal D_{\mathrm{mot}}(k)\cong \mathcal D_{\mathrm{mot}}(C)
\]
proved in \cite[Theorem 11.9]{BerkovichMotives}.
\end{proof}

\subsection{Stacks of $L$-parameters} Using the stack $\mathrm{MG}_{\Div^1}$, we can give a canonical definition of stacks of $L$-parameters, as algebraic stacks over $\mathrm{MG}_k$. Recall that the Langlands dual group $\widehat{G}$ is canonically defined as an algebraic group over $\mathrm{MG}_{\mathrm{Div}^1}$ (which incorporates the Galois action and the Tate twist implicit in geometric Satake). More precisely, it is equipped with a maximal torus and Borel $\widehat{T}\subset \widehat{B}\subset \widehat{G}$ with $X_\ast(\widehat{T})$ given by the pullback of $X^\ast(T)$ under the natural map $\mathrm{MG}_{\mathrm{Div}^1}\to \ast/\mathrm{Gal}(\overline{E}/E)$ (where as above the classifying stack of $\mathrm{Gal}(\overline{E}/E)$ is well-defined independently of a choice of $\overline{E}$). Moreover, the simple root spaces $\widehat{U}_a$ of $\widehat{G}$ are identified with the first Tate twist $L$, as in \cite[Section VI.11]{FarguesScholze}. This pins down $\widehat{G}$ uniquely up to unique isomorphism.

\begin{definition} Let $\mathrm{Par}_G$ be the algebraic stack over $\mathrm{MG}_k$ taking any ring $R$ with a map $\mathrm{Spec}(R)\to \mathrm{MG}_k$ to the groupoid of $\widehat{G}$-torsors over $\mathrm{MG}_{\Div^1}\times_{\mathrm{MG}_k} \mathrm{Spec}(R)$.
\end{definition}

Let us make this more explicit, under suitable choices. First, we can pick an isomorphism $k^\times\cong \mathbb Q/\mathbb Z[\tfrac 1p]$ yielding a cover $\mathrm{Spec}(\mathbb Z[\tfrac 1p])\to \mathrm{MG}_k$, and we describe only the pullback. Next, we pick a uniformizer $\pi\in E$, a separable closure $\overline{\breve{E}}$ and compatible roots $\pi^{1/n}\in \overline{\breve{E}}$ for $n$ prime to $p$. This yields a section
\[
\mathrm{Spec}(\mathbb Z[\tfrac 1p])\to \mathrm{MG}_{\Div^1}\times_{\mathrm{MG}_k} \mathrm{Spec}(\mathbb Z[\tfrac 1p])
\]
and writes $\mathrm{MG}_{\Div^1}\times_{\mathrm{MG}_k} \mathrm{Spec}(\mathbb Z[\tfrac 1p])$ as the classifying stack of some group scheme $\mathrm{WD}_E$ over $\mathrm{Spec}(\mathbb Z[\tfrac 1p])$. This sits naturally in a short exact sequence
\[
1\to \mathrm{ID}_E\to \mathrm{WD}_E\to \mathbb Z\to 1
\]
where the inertia--Deligne group $\mathrm{ID}_E$ again sits in an extension
\[
1\to \varprojlim_n \mathbb G_a\to \mathrm{ID}_E\to I_E\to 1
\]
of the inertia group $I_E$ by the rationalized version $\varprojlim_n \mathbb G_a$ of the additive group. This extension is naturally split over the wild inertia group $P_E\subset I_E$. Indeed, the group $C^\times/(1+C_{<1})$ depends only on the maximal tamely ramified subfield $C^t := C^{P_E}\subset C$; more precisely, the map
\[
(C^t)^\times/(1+C^t_{<1})\otimes_{\mathbb Z} \mathbb Z[\tfrac 1p]\to C^\times/(1+C_{<1})
\]
is an isomorphism. Let
\[
\mathrm{ID}_E^t = \mathrm{ID}_E/P_E
\]
be the tame quotient. This sits again in an extension
\[
1\to \varprojlim_n \mathbb G_a\to \mathrm{ID}_E^t\to I_E^t\to 1
\]
where $I_E^t\cong \prod_{\ell\neq p}\mathbb Z_\ell(1)$ is the tame inertia. Our choice of $k^\times\cong \mathbb Q/\mathbb Z[\tfrac 1p]$ trivializes all these Tate twists, and gives a distinguished generator $\tau\in I_E^t$. Our choices in fact yield a lift of this generator to $\mathrm{ID}_E^t$. In the moduli description, the given $\mathrm{Spec}(\mathbb Z[\tfrac 1p])$-point of $\mathrm{MG}_{\mathrm{Div}^1}$ comes from the extension
\[
1\to \mathbb Z[\tfrac 1p]\to M\to C^{t\times}/(1+C^t_{<1})\otimes_{\mathbb Z} \mathbb Z[\tfrac 1p]\to 1
\]
where $M=\mathbb Q\times\mathbb Q$ projecting onto
\[
C^{t\times}/(1+C^t_{<1})\otimes_{\mathbb Z} \mathbb Z[\tfrac 1p]\cong \mathbb Q\oplus \mathbb Q/\mathbb Z[\tfrac 1p]
\]
via our choice of $k^\times\cong \mathbb Q/\mathbb Z[\tfrac 1p]$ and the choice of roots $\pi^{1/n}$ of $\pi$. Then we get an automorphism of this extension, given by acting on $M$ via addition of the composite map
\[
M\to C^{t\times}/(1+C^t_{<1})\otimes_{\mathbb Z} \mathbb Z[\tfrac 1p]\to \mathbb Q\to M
\]
where the second map is the valuation map, and the third map is the fixed map $\mathbb Q\to M$ coming from our choice of roots of $\pi$ and $k^\times\cong \mathbb Q/\mathbb Z[\tfrac 1p]$.

Thus, we get a canonical injective map
\[
\mathbb Z[\tfrac 1p]\hookrightarrow \mathrm{ID}_E^t.
\]
If we fix a Frobenius lift on $C$ fixing all $\pi^{1/n}$ for $n$ prime to $p$, then this subgroup is stable under conjugation by this Frobenius lift. This yields a subgroup $\mathrm{WD}_E^{\mathrm{disc}}\subset \mathrm{WD}_E$ which sits in compatible exact sequences
\[\xymatrix{
1\ar[r] & \mathrm{ID}_E^{\mathrm{disc}}\ar[r]\ar[d] & \mathrm{WD}_E^{\mathrm{disc}}\ar[r]\ar[d] & \mathbb Z\ar[r]\ar[d] & 1\\
1\ar[r] & \mathrm{ID}_E\ar[r] & \mathrm{WD}_E\ar[r] & \mathbb Z\ar[r] & 1
}\]
and
\[\xymatrix{
1\ar[r] & P_E\ar[r]\ar[d] & \mathrm{ID}_E^{\mathrm{disc}}\ar[r]\ar[d] & \mathbb Z[\tfrac 1p]\ar[r]\ar[d] & 1\\
1\ar[r] & P_E\ar[r] & \mathrm{ID}_E\ar[r] & \mathrm{ID}_E^t\ar[r] & 1.
}\]

The group $\mathrm{WD}_E^{\mathrm{disc}}$ is precisely the version of the Weil group of $E$ with discretized inertia used in \cite{DHKM}. The following proposition shows that as far as representations go, one can replace $\mathrm{WD}_E$ by its discretization.

\begin{proposition} For any $\mathbb Z[\tfrac 1p]$-algebra $R$, pullback along $\mathrm{WD}_E^{\mathrm{disc}}\subset \mathrm{WD}_E$ yields an exact equivalence between their respective categories of representations on finite projective $R$-modules, and hence also of their maps towards $\widehat{G}$.
\end{proposition}

\begin{proof} By Tannaka, it suffices to handle the case of representations. First, we observe that we can assume that $R$ is a finitely generated $\mathbb Z$-algebra. For both the Weil--Deligne group and its discretization, representations are trivial on some open subgroup of the inertia $P_E$, writing both categories of representations as increasing unions; it suffices to handle the equivalence for their respective quotients by open subgroups of $P_E$. The resulting quotient of $\mathrm{WD}_E^{\mathrm{disc}}$ is a finitely presented group, so any representation is defined over some finitely generated $\mathbb Z$-algebra; a similar argument applies to $\mathrm{WD}_E$.

In particular, $R$ can be written as the fibre product (``arithmetic fracture square'')
\[
R=\hat{R}\times_{\hat{R}\otimes \mathbb Q} R\otimes \mathbb Q,
\]
where $\hat{R} = \mathrm{lim}_n R/n$ is the profinite completion of $R$. (In general, this holds as soon as $R$ has bounded $p^\infty$-torsion for any prime $p$.) This reduces the statement to the cases of $\hat{R}$, $\hat{R}\otimes \mathbb Q$, and $R\otimes \mathbb Q$. But the case of $\hat{R}$ further reduces to $R/n$ for all $n$. Thus, we can reduce to the cases where $R$ is either torsion, or a $\mathbb Q$-algebra. In the torsion case, the result is \cite[proof of Theorem VIII.1.3]{FarguesScholze}. It remains to handle the case where $R$ is a $\mathbb Q$-algebra. After base change to $\mathbb Q$, the group scheme $\mathrm{ID}_E$ splits as $\mathbb G_a\times I_E$, and hence $\mathrm{WD}_E$ is after base change to $\mathbb Q$ the usual Weil--Deligne group scheme. As in the first paragraph, we can assume that the representations are trivial on some fixed open subgroup of $P_E$. Allowing ourselves to replace $E$ by finite extensions and using descent, we can then reduce to the case where this subgroup is all of $P_E$. Thus, it suffices to show that, on $\mathbb Q$-algebras, the categories of representations of $\mathbb Z[\tfrac 1p]\rtimes \mathbb Z$ and $(\mathbb G_a\times \widehat{\mathbb Z}^p)\rtimes \mathbb Z$ are equivalent, where $\mathbb Z$ acts on the subgroups via multiplication by $q$. But the action of $\tau\in \mathbb Z[\tfrac 1p]$ must be conjugate to the action of $\tau^q$ which implies that all eigenvalues of $\tau$ must be roots of unity of order prime to $p$. Now the $\widehat{\mathbb Z}^p$-part keeps track of the semisimple part of $\tau$, while the $\mathbb G_a$-part takes care of the unipotent part of $\tau$.
\end{proof}

In particular, $\mathrm{Par}_G$ agrees with the space constructed in \cite{DHKM}:

\begin{corollary}\label{cor:comparisontoDHKM} After making the choices above, the base change of $\mathrm{Par}_G$ along $\mathrm{Spec}(\mathbb Z[\tfrac 1p])\to \mathrm{MG}_k$ is isomorphic to the quotient of the scheme $Z^1(\mathrm{WD}_E^{\mathrm{disc}},\widehat{G})$ of $1$-cocycles $\mathrm{WD}_E^{\mathrm{disc}}\to \widehat{G}$ by the action of $\widehat{G}$-conjugation.
\end{corollary}

By \cite{DHKM}, the scheme $Z^1(\mathrm{WD}_E^{\mathrm{disc}},\widehat{G})$ is a disjoint union of affine schemes, each of which is flat and a local complete intersection over $\mathbb Z[\tfrac 1p]$, of dimension $\mathrm{dim}(\widehat{G})$.

\section{Geometric Satake}

The final topic that needs to be adapted to the motivic formalism is the geometric Satake equivalence. Such a motivic version was first obtained by Richarz--Scholbach \cite{RicharzScholbach} and we follow their ideas. A more refined statement is due to Cass--van den Hove--Scholbach \cite{CassvandenHoveScholbach} who in particular also looked at Beilinson--Drinfeld Grassmannians.

From \cite[Section VI.1 -- VI.6]{FarguesScholze}, everything adapts without real change. We replace the condition of being universally locally acyclic with the condition of being suave; and ``locally constant with perfect fibres'' should be read as ``dualizable''. In particular, as the analogue of \cite[Definition VI.6.1]{FarguesScholze}, we define the category
\[
\mathcal D_{\mathrm{mot}}^{\mathrm{ULA}}(\Hloc_{G,S/\mathrm{Div}^d_{\mathcal Y}})\subset \mathcal D_{\mathrm{mot}}(\Hloc_{G,S/\mathrm{Div}^d_{\mathcal Y}})
\]
as the subcategory of all those sheaves whose pullback to $\mathrm{Gr}_{G,S/\mathrm{Div}^d_{\mathcal Y}}$ is bounded (i.e., supported on finitely many Schubert cells), and suave over $S$.

With this change, \cite[Proposition VI.6.4, VI.6.5, Corollary VI.6.6]{FarguesScholze} hold true, with the same proof. The statement of \cite[Corollary VI.6.7]{FarguesScholze} changes to the following, where we still write a superscript $\mathrm{ULA}$ for the condition of being suave.

\begin{proposition}\label{prop:basechangeULAHloc} For a complete algebraically closed extension $C$ of $E$ with residue field $k$, and a split reductive group $G$ over $\mathcal O_C$, the base change of the reduction functor
\[
\mathcal D_{\mathrm{mot}}^{\mathrm{ULA}}(\Hloc_{G,\mathrm{Spd} \mathcal O_C/\mathrm{Div}^1_{\mathcal Y}})\otimes_{\mathcal D_{\mathrm{mot}}(\mathcal O_C)^{\mathrm{dual}}} \mathcal D_{\mathrm{mot}}(k)^{\mathrm{dual}}\to \mathcal D_{\mathrm{mot}}^{\mathrm{ULA}}(\Hloc_{G,\mathrm{Spd} k/\mathrm{Div}^1_{\mathcal Y}})
\]
is an equivalence, and also the base change of the generic fibre functor
\[
\mathcal D_{\mathrm{mot}}^{\mathrm{ULA}}(\Hloc_{G,\mathrm{Spd} \mathcal O_C/\mathrm{Div}^1_{\mathcal Y}})\otimes_{\mathcal D_{\mathrm{mot}}(\mathcal O_C)^{\mathrm{dual}}} \mathcal D_{\mathrm{mot}}(C)^{\mathrm{dual}}\to \mathcal D_{\mathrm{mot}}^{\mathrm{ULA}}(\Hloc_{G,\mathrm{Spd} C/\mathrm{Div}^1_{\mathcal Y}})
\]
is an equivalence.
\end{proposition}

\begin{proof} The arguments in \cite[Corollary VI.6.7]{FarguesScholze} reduce these assertions to the base $\mathrm{Spd}\mathcal O_C$ with its special and generic fibre, where the suave property reduces to dualizability. But for the base, the statements are tautologies, as we put in the base change.
\end{proof}

We will be particularly interested in the case where $C$ is a completed algebraic closure of $E$.

\begin{proposition} If $C$ is a completed algebraic closure of $E$, then the reduction functor
\[
\mathcal D_{\mathrm{mot}}(\mathcal O_C)^{\mathrm{dual}}\to \mathcal D_{\mathrm{mot}}(k)^{\mathrm{dual}}
\]
is an equivalence.
\end{proposition}

\begin{proof} We can assume that $C$ is of equal characteristic and pick a splitting $k\to \mathcal O_C$. We first check that the induced functor
\[
\mathcal D_{\mathrm{mot}}(k)\to \mathcal D_{\mathrm{mot}}(\mathcal O_C)
\]
is fully faithful. For this, consider $\mathcal M_{\mathrm{arc}}(\mathcal O_{C,r})$ with fixed norm, taking the norm of $\pi$ to some $0<r<1$; this is quasicompact, and
\[
\mathcal D_{\mathrm{mot}}(\mathcal O_C)\cong \mathrm{lim}_{r<1} \mathcal D_{\mathrm{mot}}(\mathcal M_{\mathrm{arc}}(\mathcal O_{C,r})).
\]
Thus, it suffices to see that for each $r$, the functor
\[
\mathcal D_{\mathrm{mot}}(k)\to \mathcal D_{\mathrm{mot}}(\mathcal M_{\mathrm{arc}}(\mathcal O_{C,r}))
\]
is fully faithful. This reduces to computing the pushforward of the unit $1\in \mathcal D_{\mathrm{mot}}(\mathcal M_{\mathrm{arc}}(\mathcal O_{C,r}))$. But $\mathcal O_C$ is a completed filtered colimit of power series algebras, so disc invariance shows that this pushforward is trivial.

In particular,
\[
\mathcal D_{\mathrm{mot}}(k)\to \mathcal D_{\mathrm{mot}}(\mathcal O_C)
\]
is still fully faithful on dualizable objects. It remains to see that any dualizable object is in the image. For varying $r'<r$, the maps
\[
\mathcal M_{\mathrm{arc}}(\mathcal O_{C,r'})\to \mathcal M_{\mathrm{arc}}(\mathcal O_{C,r})
\]
are closed immersions whose intersection is just $\mathcal M_{\mathrm{arc}}(k)$ (with trivial norm on $k$). This implies (cf.~\cite[Lemma 10.4, 10.5]{BerkovichMotives}) that
\[
\mathrm{colim}_{r'<r} \mathcal D_{\mathrm{mot}}(\mathcal M_{\mathrm{arc}}(\mathcal O_{C,r'}))^{\mathrm{dual}}\cong \mathcal D_{\mathrm{mot}}(k)^{\mathrm{dual}}.
\]
This implies that any object of $\mathcal D_{\mathrm{mot}}(\mathcal O_C)^{\mathrm{dual}}$ becomes isomorphic to the extension of an object of $\mathcal D_{\mathrm{mot}}(k)^{\mathrm{dual}}$ after pullback to some $\mathcal M_{\mathrm{arc}}(\mathcal O_{C,r'})$. But this covers $\mathrm{Spd}(\mathcal O_C)$, and the isomorphism necessarily descends.
\end{proof}

\begin{corollary} If $C$ is a completed algebraic closure of $E$, there is an equivalence
\[
\mathcal D_{\mathrm{mot}}^{\mathrm{ULA}}(\Hloc_{G,\mathrm{Spd} k/\mathrm{Div}^1_{\mathcal Y}})\otimes_{\mathcal D_{\mathrm{mot}}(k)^{\mathrm{dual}}} \mathcal D_{\mathrm{mot}}(C)^{\mathrm{dual}}\to \mathcal D_{\mathrm{mot}}^{\mathrm{ULA}}(\Hloc_{G,\mathrm{Spd} C/\mathrm{Div}^1_{\mathcal Y}})
\]
compatible with the operations of Verdier duality, $\otimes$, $\sHom$, and $j_! j^\ast$, $j_\ast j^\ast$, $j_! j^!$, $j_\ast j^!$ where $j$ is the locally closed immersion of any Schubert cell.
\end{corollary}

\begin{proof} Use Proposition~\ref{prop:basechangeULAHloc} and simplify it using the previous proposition. All functors preserve suave objects over $\mathcal O_C$, and hence pass through the construction.
\end{proof}

The next proposition has been proved similarly in \cite[Theorem 4.8]{RicharzScholbachWitt}.

\begin{proposition} If $C$ is a completed algebraic closure of $E$, then the full subcategory
\[
\mathcal D_{MT}^{\mathrm{ULA}}(\Hloc_{G,\mathrm{Spd} C/\mathrm{Div}^1_{\mathcal Y}})\subset \mathcal D_{\mathrm{mot}}^{\mathrm{ULA}}(\Hloc_{G,\mathrm{Spd} C/\mathrm{Div}^1_{\mathcal Y}})
\]
of mixed Tate objects is stable under Verdier duality, $\otimes$, $\sHom$, and $j_! j^\ast$, $j_\ast j^\ast$, $j_! j^!$, $j_\ast j^!$ where $j$ is the locally closed immersion of any Schubert cell. It is also stable under convolution.

Moreover, the natural functor
\[
\mathcal D_{MT}^{\mathrm{ULA}}(\Hloc_{G,\mathrm{Spd} C/\mathrm{Div}^1_{\mathcal Y}})\otimes_{\mathcal D_{MT}(C)^{\mathrm{dual}}} \mathcal D_{\mathrm{mot}}(C)^{\mathrm{dual}}\to \mathcal D_{\mathrm{mot}}^{\mathrm{ULA}}(\Hloc_{G,\mathrm{Spd} C/\mathrm{Div}^1_{\mathcal Y}})
\]
is an equivalence.
\end{proposition}

\begin{proof} Stability under derived tensor products and $j_! j^\ast$ is clear from the definition. For the other operations the key is to prove stability under Verdier duality. It suffices to check this after pullback to the affine flag variety instead, as this is a fibration in classical flag varieties over the affine Grassmannian and hence pullback commutes with Verdier duality up to Tate twist. On the affine flag variety, each Schubert variety has a Demazure--Bott--Samelson resolution which is a successive $\mathbb P^1$-fibration. Taking the direct image of $\mathbb Z[\tfrac 1p]$ along this resolution, we get a sheaf which is (up to shift and Tate twist) Verdier selfdual. It suffices to see that it is also mixed Tate. But by its definition, it is a successive convolution of mixed Tate sheaves. Thus, it suffices to prove that mixed Tate sheaves on the affine flag variety are stable under convolution (which will then also imply the same stability under convolution on the affine Grassmannian). Using as basic sheaves extensions by zero from open strata, and writing everything in terms of products of simple reflections, one has to check this only for the case of twice the same simple reflection. In that case, the geometry is completely explicit: One has a $\mathbb P^1$-fibration over $\mathbb P^1$ mapping to a $\mathbb P^1$, which in fact means that the total space is a product of two $\mathbb P^1$'s. Under this isomorphism, the open stratum (an $\mathbb A^1$-fibration over $\mathbb A^1$) corresponds to $\mathbb A^1\times \mathbb P^1\setminus \Delta(\mathbb A^1)$. The $!$-pushforward towards the first projection to $\mathbb A^1$ is then mixed Tate, as desired.

The final statement is clear on each stratum, and then follows in full as the gluing functors preserve the mixed Tate subcategories.
\end{proof}

\begin{proposition} If $C$ is a completed algebraic closure of $E$, there is a perverse $t$-structure on
\[
\mathcal D_{MT}^{\mathrm{ULA}}(\Hloc_{G,\mathrm{Spd} C/\mathrm{Div}^1_{\mathcal Y}})
\]
where an object $A$ is connective if and only if for all $\mu$, the stalk at the Schubert cell parametrized by $\mu$ is an object of
\[
\mathcal D_{MT}(C)\cong \mathcal D_{\mathrm{qc}}(\mathrm{MG}_C)
\]
lying in homological degrees $\geq \langle 2\rho,\mu\rangle$.

The convolution product preserves the connective part of the $t$-structure. The heart of the perverse $t$-structure embeds fully faithfully into $\mathcal D_{MT}^{\mathrm{ULA}}(\mathrm{Gr}_{G,\mathrm{Spd} C/\mathrm{Div}^1_{\mathcal Y}})$. With $\mathbb Q$-coefficients, we have an equivalence
\[
\bigoplus_{\mu\in X_\ast(T)^+} \mathcal D_{MT}(C,\mathbb Q)^\heartsuit\otimes \mathrm{IC}_\mu\to \mathcal D_{MT}^{\mathrm{ULA}}(\Hloc_{G,\mathrm{Spd} C/\mathrm{Div}^1_{\mathcal Y}},\mathbb Q)^\heartsuit
\]
where $\mathrm{IC}_\mu$ denotes the motivic intersection complex on the Schubert variety parametrized by $\mu$.
\end{proposition}

\begin{proof} On $\mathcal D_{MT}(C)$, the $\ell$-adic realization functors
\[
\mathcal D_{\mathrm{mot}}^{\mathrm{oc}}(-)\to \mathrm{lim}_n \mathcal D_{\mathrm{mot}}^{\mathrm{oc}}(-,\mathbb Z/\ell^n)\cong  \mathrm{lim}_n \mathcal D_{\mathrm{et}}^{\mathrm{oc}}(-,\mathbb Z/\ell^n)
\]
are (jointly, for all $\ell$) conservative and $t$-exact (with respect to the motivic $t$-structure defined (only) on $\mathcal D_{MT}(C)$). Note that these $\ell$-adic realization functors also commute with all occurring functors auch as pushforwards $j_\ast$ along open immersions, as the motivic and \'etale categories are equivalent modulo $\ell^n$ (and hence even the right adjoints functors must agree). The assertions then easily follow from the known $\ell$-adic statements. For example, the construction of the connective cover can be given by Deligne's inductive formula, using that at each stratum the required truncation operator exists via reduction to the case of $\mathcal D_{MT}(C)$. The stability under convolution immediately reduces to the $\ell$-adic case. The fully faithfulness of pullback to the affine Grassmannian follows as in \cite[Lemma VI.7.3]{FarguesScholze} using in addition that for each Schubert cell, looking at the subgroup of $L^+ G$ fixing a point in the Schubert cell, its motive is mixed Tate. Finally, with $\mathbb Q$-coefficients, we get some category linear over $\mathrm{Spec}(\mathbb Q)/\mathbb G_a\rtimes \mathbb G_m$ whose base change to $\mathbb Q_\ell$, for any $\ell\neq p$, is the category of $\mathbb Q_\ell$-perverse sheaves (defined here simply as the base change to $\mathbb Q_\ell$ of the limit of the categories with $\mathbb Z/\ell^n$-coefficients). This is known to be semisimple with simple objects given by intersection complexes (cf.~\cite[Proposition VI.7.5]{FarguesScholze}). By descent along $\mathrm{Spec}(\mathbb Q_\ell)\to \mathrm{Spec}(\mathbb Q)/\mathbb G_a\rtimes \mathbb G_m$, this gives the desired description.
\end{proof}

The following theorem is due to van den Hove \cite{vandenHoveWitt}. As his argument is a rather complicated combinatorial argument, we offer a more direct argument. However, we stress that van den Hove's result is stronger than what we prove below: for example, he works with non-\'etale sheafified motives.

\begin{theorem} The subcategory
\[
\mathcal D_{MT}^{\mathrm{ULA}}(\Hloc_{G,\mathrm{Spd} C/\mathrm{Div}^1_{\mathcal Y}})\subset \mathcal D_{\mathrm{mot}}^{\mathrm{ULA}}(\Hloc_{G,\mathrm{Spd} C/\mathrm{Div}^1_{\mathcal Y}})
\]
is stable under constant term functors. More precisely, for any parabolic $P\subset G$, an object $A\in D_{\mathrm{mot}}^{\mathrm{ULA}}(\Hloc_{G,\mathrm{Spd} C/\mathrm{Div}^1_{\mathcal Y}})$ lies in the mixed-Tate subcategory if and only if $\mathrm{CT}_P(A)$ lies in the mixed-Tate subcategory.
\end{theorem}

\begin{proof} By transitivity of constant terms, it suffices to establish the claim on constant terms only in the case $P=B$ a Borel. Once the forward direction is established, the backwards direction is clear by induction on the support: Namely, the fibre at the most generic point in the support can be directly read off from a fibre of the constant term, see the proof of \cite[Proposition VI.4.2]{FarguesScholze}.

To prove that $\mathrm{CT}_B$ preserves the mixed-Tate subcategory, it suffices to prove that the composite with the projection along $\mathrm{Gr}_{T,\mathrm{Spd} C/\mathrm{Div}^1_{\mathcal Y}}\to \mathrm{Spd} C$ preserves this; we consider the resulting functor
\[
F: \mathcal D_{\mathrm{mot}}^{\mathrm{ULA}}(\Hloc_{G,\mathrm{Spd} C/\mathrm{Div}^1_{\mathcal Y}})\to \mathcal D_{\mathrm{mot}}(C)^{\mathrm{dual}}.
\]
The result is vacuous with torsion coefficients as then everything is mixed-Tate according to our definition, so we focus on rational coefficients.

As a starting point, we prove this directly for all minuscule or quasi-minuscule Schubert cells. Minuscule Schubert cells are classical flag varieties with Mirkovi\'c--Vilonen cycles given by actual affine spaces, so the claim is clear. In the quasi-minuscule case $\mathrm{Gr}_{G,\mu}$, the Mirkovi\'c--Vilonen cycles are either indexed by the Weyl group orbit of $\mu$, or $0$. But those coming from the Weyl group orbit of $\mu$ are always successive $\mathbb A^1$-fibrations, so again are fine. The remaining one is $0$. But note, by hyperbolic localization, the functor of taking cohomology admits a filtration whose associated graded are the constituents of the constant term functor. To show that the last term is also mixed-Tate, it is then sufficient to show that the cohomology of the quasi-minuscule Schubert cell is mixed-Tate. But like any Schubert cell, this is an affine fibration over a projective flag variety, yielding the claim.

Now we note that convolution of sheaves is taken under $F$ to tensor products, at least up to isomorphism (which will later be canonical, but currently a non-canonical isomorphism suffices). Namely, use the convolution Beilinson--Drinfeld Grassmannian to construct a dualizable motivic sheaf over $\mathrm{Spd} C\times \mathrm{Spd} C$ whose fibre over the diagonal is given by $F$ applied to the convolution, and whose fibres away from the diagonal is the tensor product of $F$ applied to the two factors. By spreading out dualizable motivic sheaves, one sees that they are locally constant around the diagonal (more precisely, pulled back via either projection $\mathrm{Spd} C\times \mathrm{Spd} C\to \mathrm{Spd} C$), yielding the desired claim.

But all sheaves in $\mathcal D_{MT}^{\mathrm{ULA}}(\Hloc_{G,\mathrm{Spd} C/\mathrm{Div}^1_{\mathcal Y}},\mathbb Q)$ are generated under convolution by sheaves supported on the minuscule and quasi-minuscule Schubert cells. Indeed, it suffices to generate the perverse sheaves, but those are sums of intersection complexes, and any one appears as a summand in such a convolution. Namely, as
\[
\bigoplus_{\mu\in X_\ast(T)^+} \mathcal D_{MT}(C,\mathbb Q)^\heartsuit\otimes \mathrm{IC}_\mu\to \mathcal D_{MT}^{\mathrm{ULA}}(\Hloc_{G,\mathrm{Spd} C/\mathrm{Div}^1_{\mathcal Y}},\mathbb Q)^\heartsuit,
\]
compatibly with the corresponding statement with $\mathbb Q_\ell$-coefficients, it suffices to show that in the $\ell$-adic case, any $\mathrm{IC}_\mu$ appears as a summand in such a convolution. But this follows from geometric Satake and the corresponding statement about representations of reductive groups (cf.~also \cite[Proposition 9.6]{NgoPolo}).
\end{proof}

For any finite set $I$, we can now consider
\[
\mathcal D_{\mathrm{mot}}^{\mathrm{oc}}(\Hloc_G^I) = \mathcal D_{\mathrm{mot}}^{\mathrm{oc}}(\Hloc_{G,\mathrm{Div}^I}),
\]
where as above $\mathrm{Div}^I := (\mathrm{Div}^1)^I$. Inside it, we have the subcategory
\[
\mathcal D_{\mathrm{mot}}^{\mathrm{ULA}}(\Hloc_G^I)\subset \mathcal D_{\mathrm{mot}}^{\mathrm{oc}}(\Hloc_G^I)
\]
of bounded complexes that are suave over $\mathrm{Div}^I$, and we can further restrict to mixed Tate sheaves
\[
\mathcal D_{MT}^{\mathrm{ULA}}(\Hloc_G^I)\subset \mathcal D_{\mathrm{mot}}^{\mathrm{ULA}}(\Hloc_G^I).
\]
This category is still stable under convolution. Also, on this subcategory, the $\ell$-adic realization functors
\[
\mathcal D_{\mathrm{mot}}(-)\to \mathrm{lim}_n \mathcal D_{\mathrm{et}}(-,\mathbb Z/\ell^n)
\]
are conservative. (Note that $\mathrm{Div}^I$ has a dense set of $C$-points where $C$ is the completed algebraic closure of $E$.) There is thus at most one $t$-structure making the $\ell$-adic realization functors $t$-exact for the relative perverse $t$-structure constructed in \cite{FarguesScholze}. We claim that this $t$-structure exists. To see this, we enlarge the category, and work with the full subcategory
\[
\mathcal D_{MT}^{\mathrm{ULA},\mathrm{strat}}(\Hloc_G^I)\subset \mathcal D_{\mathrm{mot}}^{\mathrm{oc}}(\Hloc_G^I)
\]
that is generated by the images of the fully faithful functors
\[
\mathcal D_{MT}^{\mathrm{ULA}}(\Hloc_G^J)\subset \mathcal D_{\mathrm{mot}}^{\mathrm{oc}}(\Hloc_G^J)\subset \mathcal D_{\mathrm{mot}}^{\mathrm{oc}}(\Hloc_G^I)
\]
for arbitrary surjections $I\to J$ (corresponding to sheaves pushed forward from partial diagonals $\mathrm{Div}^J\subset \mathrm{Div}^I$). This category is still stable under Verdier duality (note that Verdier duality in the absolute sense and relatively over some $\mathrm{Div}^I$ agree up to shift and Tate twist). If $j: U\hookrightarrow \mathrm{Div}^I$ is an open subset whose complement is a union of partial diagonals, then the functors $j_! j^\ast$ and $j_\ast j^\ast$ preserve $\mathcal D_{MT}^{\mathrm{ULA},\mathrm{strat}}(\Hloc_G^I)$. Indeed, the second functor reduces to the first by Verdier duality, and the first functor can be expressed using triangles in terms of the functors of the form $i_\ast i^\ast$ where $i: \mathrm{Div}^J\hookrightarrow \mathrm{Div}^I$ is a closed immersion. Now it is immediate from the definition of the category $\mathcal D_{MT}^{\mathrm{ULA},\mathrm{strat}}(\Hloc_G^I)$.

We claim now that $\mathcal D_{MT}^{\mathrm{ULA},\mathrm{strat}}(\Hloc_G^I)$ has a, necessarily unique, relatively perverse $t$-structure over $\mathrm{Div}^I$, for which the $\ell$-adic realization functors are $t$-exact. This can be proved by induction on open subsets $U\subset \mathrm{Div}^I$ as above, using that on any stratum such a $t$-structure does exist. Moreover, this $t$-structure passes to the subcategory
\[
\mathcal D_{MT}^{\mathrm{ULA}}(\Hloc_G^I)\subset \mathcal D_{MT}^{\mathrm{ULA},\mathrm{strat}}(\Hloc_G^I):
\]
As we know this on $\ell$-adic realizations, it suffices to see that an object $M$ of the heart of $\mathcal D_{MT}^{\mathrm{ULA},\mathrm{strat}}(\Hloc_G^I)$ whose $\ell$-adic realization is suave over $\mathrm{Div}^I$, is already suave over $\mathrm{Div}^I$. To see this, let $j: U\hookrightarrow \mathrm{Div}^I$ be the complement of all partial diagonals, and let $\tilde{j}$ denotes its pullback to $\Hloc_G^I$; then we claim that
\[
M\to {}^p\mathcal H^0(\tilde{j}_\ast \tilde{j}^\ast M)
\]
is an isomorphism. Both sides are objects of $\mathcal D_{MT}^{\mathrm{ULA},\mathrm{strat}}(\Hloc_G^I)$ and we know that it is an isomorphism after $\ell$-adic realizations; thus, it is an isomorphism. Now we claim that for any object $M$ in the heart of $\mathcal D_{MT}^{\mathrm{ULA},\mathrm{strat}}(\Hloc_G^I)$, the object ${}^p\mathcal H^0(\tilde{j}_\ast \tilde{j}^\ast M)$ is suave over $\mathrm{Div}^I$. As this only depends on $j^\ast M$, and $M\mapsto j^\ast M$ kills all objects coming from partial diagonals, we can assume that $M$ is suave over $\mathrm{Div}^I$. In this case it suffices to see that
\[
M\to {}^p\mathcal H^0(\tilde{j}_\ast \tilde{j}^\ast M)
\]
is an isomorphism. Again, this is true on $\ell$-adic realizations, so follows.

The preceding argument also shows that the functor $j^\ast$ is fully faithful on the heart of $\mathcal D_{MT}^{\mathrm{ULA}}(\Hloc_G^I)$. We can also restrict to the flat (over the coefficients $\mathbb Z[\tfrac 1p]$) perverse sheaves, finally leading to
\[
\mathrm{Sat}_G^I\subset \mathcal D_{MT}^{\mathrm{ULA}}(\Hloc_G^I).
\]
This is an exact monoidal category (via convolution) that is linear over the category of vector bundles on $\mathrm{MG}_{\mathrm{Div}^I}$. We can then repeat the discussion in \cite{FarguesScholze} to identify it. We note that we already showed that the functors $j^\ast$ are fully faithful, which makes it possible to repeat the construction of the fusion symmetric monoidal structure, commuting with (and hence refining, by Eckmann--Hilton) the convolution monoidal structure.

Let $\widehat{G}$ over $\mathrm{MG}_{\mathrm{Div}^1}$ be the cyclotomically twisted pinned reductive group with root datum dual to $G$. As explained above, the cyclotomically twisted pinning means that the root spaces are identified with $\mathbb Z[\tfrac 1p](1)$, not $\mathbb Z[\tfrac 1p]$. For any finite set $I$, let $\mathrm{MG}_{\mathrm{Div}^I} = \mathrm{MG}_{\mathrm{Div}^1}^{I/\mathrm{MG}_k}$ and $\widehat{G}^I$ be the corresponding reductive group over $\mathrm{MG}_{\mathrm{Div}^I}$.

\begin{theorem} Functorially in the finite set $I$, the category $\mathrm{Sat}_G^I$ is equivalent to the category of representations of the group scheme $\widehat{G}^I$ on vector bundles over $\mathrm{MG}_{\mathrm{Div}^I}$.
\end{theorem}

We also have the analogous results on compatibility with constant terms, and with Chevalley involutions.

\begin{proof} One follows the arguments in \cite[Sections VI.9 -- VI.11]{FarguesScholze}. The starting point is the version of Drinfeld's lemma, cf.~Theorem~\ref{thm:drinfeldlemma}, showing that vector bundles on $\mathrm{MG}_{\mathrm{Div}^I}$ are equivalent to heart of the above $t$-structure on mixed-Tate local systems on $\mathrm{Div}^I$. The fibre functor on $\mathrm{Sat}_G^I$ is given as usual via the direct sum $\bigoplus_{i\in \mathbb Z} R^i\pi_\ast$ where $\pi: \mathrm{Gr}_G^I\to \mathrm{Div}^I$ is the projection. If one fixes $T\leftarrow B\subset G$ (possibly after enlarging $E$, or more canonically by doing it in the universal family of $B$'s, i.e.~over the flag variety), this can also be described in terms of hyperbolic localization. In fact, as in \cite[Proposition VI.9.6]{FarguesScholze}, hyperbolic localization induces symmetric monoidal functor between Satake categories for a group $G$ and a Levi $M$, compatibly with the fibre functors.

Following \cite[Section VI.10]{FarguesScholze}, one can then write $\mathrm{Sat}_G^I$ as representations of a Hopf algebra $\mathcal H_G^I$ in vector bundles on $\mathrm{MG}_{\mathrm{Div}^I}$. The key existence of projective envelopes in bounded weight categories, \cite[Proposition VI.10.1]{FarguesScholze}, can be proved with the same argument; in fact it is slightly easier as we directly work with $\mathbb Z$-coefficients, for which the vanishing of the perverse sheaf $K$ occurring in the proof can be deduced from its vanishing with $\mathbb Q$-coefficients, where it follows from semisimplicity of the category. Moreover, we also know that $\mathcal H_G^I = \bigotimes_{i\in I} \mathcal H_G^{\{i\}}$, so we are reduced to the case $I=\ast$.

It remains to identify the group scheme $\widecheck{G}$ corresponding to the Hopf algebra $\mathcal H_G = \mathcal H_G^\ast$, following \cite[Section VI.11]{FarguesScholze}. After pullback along $\mathrm{Spec}(\mathbb Z_\ell)\to \mathrm{MG}_{\mathrm{Div}^1}$, we know that it is given by $\widehat{G}$. In particular, by faithfully flat descent, we know that $\widecheck{G}$ is reductive. As the fibre functor is graded, it also comes with a distinguished $\widecheck{T}\subset \widecheck{B}\subset \widecheck{G}$ which, again by faithfully flat descent from the known $\ell$-adic case, are a maximal torus and Borel. The root spaces come with a canonical identification with a Tate twist, by reduction to the case of rank $1$ groups, or more precisely $\mathrm{SL}_2$. This shows that $\widecheck{G}\cong \widehat{G}$ as desired.

We note that \cite[Proposition VI.12.1]{FarguesScholze} about Chevalley involutions also holds in the motivic case, as it just a condition, which can be checked after $\ell$-adic realization.
\end{proof}

\section{Synopsis}

In this section, we collect all the pieces. As above, we fix a reductive group $G$ over the nonarchimedean local field $E$ with residue field $\mathbb F_q$, and an algebraic closure $k=\overline{\mathbb F}_q$.

On the spectral side, we have the base stack $\mathrm{MG}_k/\mathbb Z[\tfrac 1p]$, and the dual group $\widehat{G}$ over $\mathrm{MG}_{\mathrm{Div}^1}$. We get the algebraic stack
\[
\mathrm{Par}_G = \mathrm{Hom}_{\mathrm{MG}_k}(\mathrm{MG}_{\mathrm{Div}^1},\ast/\widehat{G}).
\]
over $\mathrm{MG}_k$. After base change along a cover $\mathrm{Spec}(\mathbb Z[\tfrac 1p])\to \mathrm{MG}_k$, this is the stack of $L$-parameters over $\mathbb Z[\tfrac 1p]$ as discussed in \cite{DHKM}, by Corollary~\ref{cor:comparisontoDHKM}.

We have the following analogue of \cite[Theorem X.0.2]{FarguesScholze}.

\begin{theorem}\label{thm:constructspectralaction} Let $m=|\pi_0(Z(G))|$. Let $\mathcal C$ be a $\mathrm{Perf}(\mathrm{MG}_k)[\frac 1m]$-linear idempotent complete small stable $\infty$-category. Giving a $\mathrm{Perf}(\mathrm{MG}_k)$-linear action of $\mathrm{Perf}(\mathrm{Par}_G)$ on $\mathcal C$ is equivalent to giving, functorially in finite sets $I$, a monoidal exact $\mathrm{Vect}(\mathrm{MG}_{\mathrm{Div}^I})$-linear functor
\[
\mathrm{Rep}_{\mathrm{MG}_{\mathrm{Div}^I}}(\widehat{G}^I)\to \mathrm{End}_{\mathrm{Perf}(\mathrm{MG}_k)}(\mathcal C)\otimes_{\mathrm{Perf}(\mathrm{MG}_k)} D_{\mathrm{qc}}(\mathrm{MG}_{\mathrm{Div}^I})^\omega.
\]

Slightly more generally, assume that $\mathcal C$ is $\mathcal C_0$-linear for some symmetric monoidal idempotent-complete stable $\infty$-category $\mathcal C_0$ with a symmetric monoidal functor $\mathrm{Perf}(\mathrm{MG}_k)[\tfrac 1m]\to \mathcal C_0$. Then giving a $\mathcal C_0$-linear action of $\mathrm{Perf}(\mathrm{Par}_G)\otimes_{\mathrm{Perf}(\mathrm{MG}_k)} \mathcal C_0$ is equivalent to giving, functorially in finite sets $I$, a monoidal exact $\mathrm{Vect}(\mathrm{MG}_{\mathrm{Div}^I})$-linear functor
\[
\mathrm{Rep}_{\mathrm{MG}_{\mathrm{Div}^I}}(\widehat{G}^I)\to \mathrm{End}_{\mathcal C_0}(\mathcal C)\otimes_{\mathrm{Perf}(\mathrm{MG}_k)} D_{\mathrm{qc}}(\mathrm{MG}_{\mathrm{Div}^I})^\omega.
\]
\end{theorem}

\begin{proof} Indeed, there is a natural map in one direction, from actions of $\mathrm{Perf}(\mathrm{Par}_G)$ towards such maps of monoidal functors functorially in $I$. To see that this is an equivalence, we may work on the cover $\mathrm{Spec}(\mathbb Z[\tfrac 1p])\to \mathrm{MG}_k$, by descent in the variable $\mathcal C_0$. After this cover, we have the discretization $\mathrm{WD}_E^{\mathrm{disc}}$ of the Weil--Deligne group, and everything can be described in terms of representations of this group. More precisely, for each $I$ the functor
\[
D_{\mathrm{qc}}(\mathrm{MG}_{\mathrm{Div}^I}\times_{\mathrm{MG}_k} \mathrm{Spec}(\mathbb Z[\tfrac 1p]))\cong D_{\mathrm{qc}}(\ast/\mathrm{WD}_E^I)\to D_{\mathrm{qc}}(\ast/(\mathrm{WD}_E^{\mathrm{disc}})^I)
\]
is fully faithful and preserves compact objects. This yields fully faithful functors
\[\begin{aligned}
\mathrm{End}_{\mathcal C_0}(\mathcal C)\otimes_{\mathrm{Perf}(\mathrm{MG}_k)} D_{\mathrm{qc}}(\mathrm{MG}_{\mathrm{Div}^I})^\omega\cong &\ \mathrm{End}_{\mathcal C_0}(\mathcal C)\otimes D_{\mathrm{qc}}(\ast/\mathrm{WD}_E^I)^\omega\\
\hookrightarrow &\ \mathrm{End}_{\mathcal C_0}(\mathcal C)\otimes D_{\mathrm{qc}}(\ast/(\mathrm{WD}_E^{\mathrm{disc}})^I)^\omega.
\end{aligned}\]
The target contains $\mathrm{End}_{\mathcal C_0}(\mathcal C)^{(\mathrm{WD}_E^{\mathrm{disc}})^I}$ fully faithfully, and in fact monoidality guarantees that the image of $\mathrm{Rep}_{\mathrm{MG}_{\mathrm{Div}^I}}(\widehat{G}^I)$ must land inside this subcategory (as all dualizable objects are in there).

Now the same proof as in \cite[Section X.3]{FarguesScholze} applies to show the equivalence of this, a priori larger, class of data with $\mathbb Z[\tfrac 1{pm}]$-linear actions of the symmetric monoidal stable $\infty$-category $\mathrm{Perf}(\mathrm{Par}_G\times_{\mathrm{MG}_k} \mathrm{Spec}(\mathbb Z[\tfrac 1{pm}]))$. As actions by such already yield data factoring over $D_{\mathrm{qc}}(\ast/\mathrm{WD}_E^I)\subset D_{\mathrm{qc}}(\ast/(\mathrm{WD}_E^{\mathrm{disc}})^I)$, we get the desired equivalence.
\end{proof}

On the geometric side, we have the stack $\mathrm{Bun}_G$, and the $\mathcal D_{\mathrm{mot}}(k)$-linear category
\[
\mathcal D_{\mathrm{mot}}(\mathrm{Bun}_G).
\]
For any finite set $I$, the geometric Satake equivalence and Hecke operators
\[\xymatrix{
& \mathrm{Hck}_G^I\ar[dl]_{p_1}\ar[dr]^{p_2}\ar[rr]^\epsilon && \Hloc_G^I\\
\mathrm{Bun}_G && \mathrm{Bun}_G\times \mathrm{Div}^I
}\]
yield monoidal exact $\mathrm{Vect}(\mathrm{MG}_{\mathrm{Div}^I})$-linear functors
\[
\mathrm{Rep}_{\mathrm{MG}_{\mathrm{Div}^I}}(\widehat{G}^I)\cong \mathrm{Sat}_G^I\to \mathrm{End}_{\mathcal D_{\mathrm{mot}}(k)}(\mathcal D_{\mathrm{mot}}(\mathrm{Bun}_G))\otimes_{D_{\mathrm{qc}}(\mathrm{MG}_k)} D_{\mathrm{qc}}(\mathrm{MG}_{\mathrm{Div}^I})
\]
via taking $V$ to the functor
\[
T_V = p_{2\ast}(p_1^\ast(-)\otimes \epsilon^\ast \mathrm{Sat}(V)): \mathcal D_{\mathrm{mot}}(\mathrm{Bun}_G)\to \mathcal D_{\mathrm{mot}}(\mathrm{Bun}_G\times \mathrm{Div}^I).
\]
Here, the Hecke operators a priori take values in
\[
\mathcal D_{\mathrm{mot}}(\mathrm{Bun}_G\times \mathrm{Div}^I)\cong \mathcal D_{\mathrm{mot}}(\mathrm{Bun}_G)\otimes_{\mathcal D_{\mathrm{mot}}(k)} \mathcal D_{\mathrm{mot}}((\Div^1)^I),
\]
but actually take image in the tensor product with the subcategory
\[
\mathcal D_{\mathrm{mot}}(\Div^1)^{\otimes_{\mathcal D_{\mathrm{mot}}(k)} I}\subset \mathcal D_{\mathrm{mot}}((\Div^1)^I),
\]
by writing Hecke operators with several legs as composites of Hecke operators with only one leg. Taking $\mathcal C_0=\mathcal D_{\mathrm{mot}}(k)^\omega$ and $\mathcal C=\mathcal D_{\mathrm{mot}}(\mathrm{Bun}_G)^\omega$, and noting that $\mathcal C$ is preserved under Hecke operators (by monoidality of the Hecke action, which means that the left and right adjoints also exist and are also colimit-preserving Hecke operators), we get the desired data to apply Theorem~\ref{thm:constructspectralaction}.

Putting everything together yields the spectral action.

\begin{theorem} Denoting $m=|\pi_0(Z(G))|$, the above constructions yield a $\mathcal D_{\mathrm{mot}}(k)^\omega[\tfrac 1m]$-linear action of
\[
\mathrm{Perf}(\mathrm{Par}_G)\otimes_{\mathrm{Perf}(\mathrm{MG}_k)} \mathcal D_{\mathrm{mot}}(k)^\omega[\tfrac 1m]
\]
on
\[
\mathcal D_{\mathrm{mot}}(\Bun_G)^\omega[\tfrac 1m],
\]
extending the Hecke action $V\mapsto T_V$ above.
\end{theorem}

\begin{corollary} The independence of $\ell$ conjecture \cite[Conjecture I.9.5]{FarguesScholze} holds true.
\end{corollary}

\begin{proof} By using z-extensions, and the compatibility of the constructions of the map from the spectral Bernstein center to the usual Bernstein center with central isogenies, one can reduce to the case that $G$ has connected center and hence $m=1$.

The motivic spectral action maps the structure sheaf to the identity endofunctor. Passing to endomorphisms, this shows that the global sections of the structure sheaf on $\mathrm{Par}_G$ map towards the Bernstein center of $\mathcal D_{\mathrm{mot}}(\mathrm{Bun}_G)$ and hence towards the Bernstein center of the open stratum, which yields the category of $G(E)$-representations on objects of $\mathcal D_{\mathrm{mot}}(k)$. This contains the category of representations in $\mathcal D(\mathbb Z[\tfrac 1p])$ fully faithfully as the endomorphisms of the unit in $\mathcal D_{\mathrm{mot}}(k)$ are $\mathbb Z[\tfrac 1p]$. Thus, its Bernstein center is just the usual Bernstein center. In summary, we get a map
\[
\mathcal O(\mathrm{Par}_G)\to \mathcal{Z}(G(E),\mathbb Z[\tfrac 1p]).
\]
Recall that $\mathrm{Par}_G$ lives over $\mathrm{MG}_k$ which is a classifying space for the group scheme $\mathbb G_m^{\mathrm{rat}}$ that is $\mathbb G_m$ over $\mathrm{Spec}(\mathbb Q)$ but trivial modulo $n$ for $n$ prime to $p$. We note if $f: \mathrm{Par}_G\to \mathrm{MG}_k$ denotes the projection, then $f_\ast \mathcal O$ lies in the full subcategory $\mathcal D(\mathbb Z[\tfrac 1p])\subset \mathcal D_{\mathrm{qc}}(\mathrm{MG}_k)$. Indeed, this is only a nontrivial statement after base change to $\mathbb Q$, and here the coarse quotient of $\mathrm{Par}_G$ is also the coarse moduli space of continuous representations of $W_E$, cf.~\cite[Proposition 4.17 (ii)]{DHKM}, and the latter cannot involve nontrivial weights as the quotient $W_E$ of $\mathrm{WD}_E$ is defined already before passing to $\mathrm{MG}_k$. It follows that $\mathcal O(\mathrm{Par}_G)$, which is defined as the global sections of $f_\ast \mathcal O\in \mathcal D_{\mathrm{qc}}(\mathrm{MG}_k)$, also agrees with the pullback of $f_\ast \mathcal O$ along any section $\mathrm{Spec}(\mathbb Z[\tfrac 1p])\to \mathrm{MG}_k$, and hence agrees with the algebra constructed in \cite[Theorem 4.18]{DHKM}.

After base change along the \'etale realization $\mathcal D_{\mathrm{mot}}(k)\to \mathcal D(\mathbb Z/\ell^n)$, everything reduces to the \'etale formalism used in \cite{FarguesScholze}. The square root of $q$ arises by trivializing the Tate twist inherent in the dual group $\widehat{G}$ over $\mathrm{MG}_k$. This shows that the displayed map above specializes to the $\ell$-adic constructions from \cite{FarguesScholze} modulo any power $\ell^n$ of $\ell$ (although this is not strictly required, we note that the left-hand side commutes with base change by \cite[Theorem VIII.3.6]{FarguesScholze} --- a priori only on any finite union of connected components, but reduction modulo $\ell^n$ also commutes with infinite products). To show agreement even with $\mathbb Z_\ell$-coefficients, it remains to observe that $\mathcal{Z}(G(E),\mathbb Z_\ell)$ is $\ell$-adically separated. But for any open pro-$p$-subgroup $K\subset G(E)$, the Hecke algebra $\mathbb Z_\ell[K\backslash G(E)/K]$ is clearly $\ell$-adically separated, and $\mathcal Z(G(E),\mathbb Z_\ell)$ injects into the product of them (or even their centers).
\end{proof}

\bibliographystyle{amsalpha}
\bibliography{MotivicGeometrization}

\newcommand{\etalchar}[1]{$^{#1}$}
\providecommand{\bysame}{\leavevmode\hbox to3em{\hrulefill}\thinspace}
\providecommand{\MR}{\relax\ifhmode\unskip\space\fi MR }
\providecommand{\MRhref}[2]{%
  \href{http://www.ams.org/mathscinet-getitem?mr=#1}{#2}
}
\providecommand{\href}[2]{#2}
\begin{thebibliography}{ABlB{\etalchar{+}}25}

\bibitem[ABlB{\etalchar{+}}25]{DeRhamFarguesFontaine}
J.~Ansch\"utz, G.~Bosco, A.-C. le~Bras, J.~E. Rodr\'iguez~Camargo, and
  P.~Scholze, \emph{Analytic de {R}ham stacks of {F}argues--{F}ontaine curves},
  \url{https://arxiv.org/abs/2510.15196}, 2025.

\bibitem[AGV22]{AyoubGallauerVezzani}
J.~Ayoub, M.~Gallauer, and A.~Vezzani, \emph{The six-functor formalism for
  rigid analytic motives}, Forum Math. Sigma \textbf{10} (2022), Paper No. e61,
  182.

\bibitem[BGV23]{BindaGallauerVezzani}
F.~Binda, M.~Gallauer, and A.~Vezzani, \emph{Motivic monodromy and {$p$}-adic
  cohomology theories}, \url{https://arXiv.org/abs/2306.05099}, 2023.

\bibitem[CvdHS22]{CassvandenHoveScholbach}
R.~Cass, T.~van~den Hove, and J.~Scholbach, \emph{The geometric {S}atake
  equivalence for integral motives}, \url{https://arXiv.org/abs/2211.04832},
  2022.

\bibitem[DHKM25]{DHKM}
J.-F. Dat, D.~Helm, R.~Kurinczuk, and G.~Moss, \emph{Moduli of {L}anglands
  parameters}, J. Eur. Math. Soc. (JEMS) \textbf{27} (2025), no.~5, 1827--1927.

\bibitem[DK24]{DauserKuijper}
A.~Dauser and J.~Kuijper, \emph{Uniqueness of six-functor formalisms},
  \url{https://arxiv.org/abs/2412.15780}, 2024.

\bibitem[FS21]{FarguesScholze}
L.~Fargues and P.~Scholze, \emph{Geometrization of the local {L}anglands
  correspondence}, arXiv:2102.13459, to appear in {A}st{\'e}risque, 2021.

\bibitem[GHW22]{GulottaHansenWeinstein}
D.~Gulotta, D.~Hansen, and J.~Weinstein, \emph{An enhanced six-functor
  formalism for diamonds and v-stacks}, \url{https://arXiv.org/abs/2202.12467},
  2022.

\bibitem[HI25]{HamannImai}
L.~Hamann and N.~Imai, \emph{Dualizing complexes on the moduli of parabolic
  bundles}, J. Reine Angew. Math. \textbf{825} (2025), 139--183.

\bibitem[HM24]{HeyerMann}
C.~Heyer and L.~Mann, \emph{$6$-{F}unctor {F}ormalisms and {S}mooth
  {R}epresentations}, arXiv:2410.13038, 2024.

\bibitem[Laf18]{VLafforgueExcursion}
V.~Lafforgue, \emph{Chtoucas pour les groupes r\'eductifs et param\'etrisation
  de {L}anglands globale}, J. Amer. Math. Soc. \textbf{31} (2018), no.~3,
  719--891.

\bibitem[LZ24]{LiuZhengArtin}
Y.~Liu and W.~Zheng, \emph{Enhanced six operations and base change theorems for
  artin stacks}, \url{https://arXiv.org/abs/1211.5948v4}, 2024.

\bibitem[Man22]{MannThesis}
L.~Mann, \emph{A {$p$}-adic {$6$}-{F}unctor {F}ormalism in {R}igid-{A}nalytic
  {G}eometry}, \url{https://arXiv.org/abs/2206.02022}, 2022.

\bibitem[NP01]{NgoPolo}
B.~C. Ng\^o and P.~Polo, \emph{R\'esolutions de {D}emazure affines et formule
  de {C}asselman-{S}halika g\'eom\'etrique}, J. Algebraic Geom. \textbf{10}
  (2001), no.~3, 515--547.

\bibitem[NY19]{NadlerYun}
D.~Nadler and Z.~Yun, \emph{Spectral action in {B}etti geometric {L}anglands},
  Israel J. Math. \textbf{232} (2019), no.~1, 299--349.

\bibitem[RS21a]{RicharzScholbach}
T.~Richarz and J.~Scholbach, \emph{The motivic {S}atake equivalence}, Math.
  Ann. \textbf{380} (2021), no.~3-4, 1595--1653.

\bibitem[RS21b]{RicharzScholbachWitt}
\bysame, \emph{Tate motives on {W}itt vector affine flag varieties}, Selecta
  Math. (N.S.) \textbf{27} (2021), no.~3, Paper No. 44, 34.

\bibitem[Sch17]{ECoD}
P.~Scholze, \emph{{\'E}tale cohomology of diamonds},
  \url{https://arXiv.org/abs/1709.07343}, to appear in {A}st{\'e}risque, 2017.

\bibitem[Sch24]{BerkovichMotives}
\bysame, \emph{Berkovich {M}otives}, \url{https://arXiv.org/abs/2412.03382}, to
  appear in Journal of the AMS, 2024.

\bibitem[Sch25]{ScholzeSixFunctors}
\bysame, \emph{Six-{F}unctor {F}ormalisms}, lecture notes,
  \url{https://arXiv.org/abs/2510.26269}, 2025.

\bibitem[vdH24]{vandenHoveWitt}
T.~van~den Hove, \emph{The integral motivic {S}atake equivalence for ramified
  groups}, \url{https://arXiv.org/abs/2404.15694}, 2024.

\end{thebibliography}

\end{document}